\documentclass[11pt]{amsart}
\usepackage{euscript}
\usepackage{amssymb}
\usepackage{amsmath}
\usepackage{epic}
\usepackage{graphics}
\usepackage{epsfig}
\usepackage{color}

\usepackage{amscd,euscript}
\usepackage[frame,cmtip,curve,arrow,matrix,line,graph]{xy}

\numberwithin{equation}{section}
\newtheoremstyle{my}{1.5em}{0.5em}{\em}{}{\sc}{.}{0.5em}{}

\newtheorem{thm}{Theorem}[section]
\newtheorem{Theorem}[thm]{Theorem}
\newtheorem*{Theorem*}{Theorem}
\newtheorem{Corollary}[thm]{Corollary}

\newtheorem*{corollary*}{Corollary}
\newtheorem{Lemma}[thm]{Lemma}

\newtheorem*{conjecture*}{Conjecture}

\newtheorem*{question*}{Question}

\newtheorem{Definition}[thm]{Definition}

\newtheorem*{definitions*}{Definitions}

\newtheorem*{rem*}{Remark}
\newtheorem{Remark}[thm]{Remark}

\newtheorem*{remark*}{Remark}

\newtheorem*{remarks*}{Remarks}
\newtheorem*{example*}{Example}
\newtheorem{Example}[thm]{Example}
\newtheorem*{examples*}{Examples}

\newtheorem*{convention*}{Convention}
\newtheorem*{conventions*}{Conventions}

\newtheorem*{exercise*}{Exercise}
\newtheorem*{bibliographical-note*}{Bibliographical note}

\newcommand{\Acknowledgements}{{\em Acknowledgements.} }

\newcommand{\scrA}{\EuScript{A}}

\newcommand{\scrX}{\EuScript{X}}
\newcommand{\scrK}{\EuScript{K}}
\newcommand{\scrY}{\EuScript{Y}}
\newcommand{\scrC}{\EuScript{C}}

\newcommand{\calO}{\mathcal{O}}

\newcommand{\bR}{\mathbb{R}}
\newcommand{\bZ}{\mathbb{Z}}

\newcommand{\bC}{\mathbb{C}}

\newcommand{\bP}{\mathbb{P}}

\newcommand{\cdbar}{\mathrm{\overline{\partial}}}
\newcommand{\Sym}{\mathrm{Sym}}

\newcommand{\id}{\mathrm{id}}

\renewcommand{\ker}{\mathrm{ker}}

\newcommand{\Hom}{\mathrm{Hom}}

\renewcommand{\index}{\mathrm{index} \, }
\newcommand{\Zer}{\mathrm{Zer}}

\newcommand{\Aut}{\mathrm{Aut}}

\newcommand{\Symp}{\mathrm{Symp}}

\newcommand{\scrM}{\EuScript{M}}

\newcommand{\scrF}{\EuScript{F}}
\newcommand{\scrJ}{\EuScript{J}}
\newcommand{\scrG}{\EuScript{G}}

\newcommand{\A}{\mathcal{A}}

\numberwithin{equation}{section}
\renewcommand{\leq}{\leqslant}
\renewcommand{\geq}{\geqslant}
\newcommand{\isom}{\cong}
\newcommand{\tensor}{\otimes}

\newcommand{\lra}{\longrightarrow}

\newcommand{\blank}{-}
\newcommand{\D}{{\mathcal{D}}}

\newcommand{\PP}{\operatorname{\mathbb P}}

\renewcommand{\L}{\operatorname{L}}

\newcommand{\CC}{\mathcal C}
\newcommand{\OO}{\mathcal O}

\newcommand{\Z}{\mathbb{Z}}

\newcommand{\Tw}{\operatorname{Tw}}

\renewcommand{\sc}{\operatorname{sc}}

\newcommand{\Quad}{\operatorname{Quad}}
\newcommand{\Stab}{\operatorname{Stab}}

\newcommand{\CY}{CY$_3$ }
\newcommand{\M}{\mathbb{M}}
\renewcommand{\S}{\mathbb{S}}
\newcommand{\T}{{T}}

\renewcommand{\Im}{\operatorname{Im}}

\newcommand{\Pol}{\mathrm{Pol}}

\title{Quiver algebras as Fukaya categories}
\author{Ivan Smith}
\thanks{The author was partially supported by grant ERC-2007-StG-205349 from the European Research Council.}
\address{Ivan Smith, Centre for Mathematical Sciences, University of Cambridge, England.}

\begin{document}
\maketitle \thispagestyle{empty}

\parindent0em
\parskip1em

\begin{abstract} We embed triangulated categories defined by quivers with potential arising from ideal triangulations of marked bordered surfaces into Fukaya categories of quasi-projective 3-folds associated to meromorphic quadratic differentials.  
Together with previous results, this yields non-trivial computations of spaces of stability conditions on  Fukaya categories of symplectic six-manifolds.
\end{abstract}





\section{Introduction}

A marked bordered surface $(\S,\M)$ comprises a compact, connected oriented surface $\S$, perhaps with non-empty boundary, together with a non-empty set $\M \subset \S$ of marked points, such that every boundary component of $\S$ contains at least one marked point.  We always assume that $(\S,\M)$ is not a sphere with fewer than five marked points.   An ideal triangulation of  $(\S,\M)$ gives rise, via work of Labardini-Fragoso \cite{LF1}, to a quiver with potential.  The \CY-triangulated category $\D(\S,\M)$ of finite-dimensional modules over the corresponding Ginzburg algebra depends only on the underlying data $(\S,\M)$. This paper embeds these categories, under mild  hypotheses on $(\S,\M)$, into Fukaya categories of quasi-projective 3-folds. The 3-folds are the total spaces of affine conic fibrations over $\S$; whilst these spaces have not appeared previously in the literature, they are close cousins of those studied in \cite{DDDHP}. Together with the main results of \cite{BrSm}, we therefore obtain  computations of spaces of stability conditions on (distinguished subcategories of) Fukaya categories of symplectic six-manifolds.

We work over an algebraically closed field $k$ of characteristic zero.  For much of the paper, we take $k$ to be the single variable Novikov field (with formal parameter $q$), 
\begin{equation} \label{Eqn:NovikovField}
\Lambda_{\bC} \ = \ \left\{ \sum c_i q^{m_i} \ \big| \ c_i \in \bC, \, m_i \in \bR, \, \lim_{i\rightarrow \infty} m_i = +\infty\right\}
\end{equation}
which is algebraically closed by \cite[Lemma 13.1]{FO3:toric}.

\subsection{Surfaces and differentials}

Let $(\S,\M)$ be a marked bordered surface.  An ideal triangulation $T$ of $\S$ with vertices at $\M$ has an associated quiver with potential $(Q(T), W(T))$, which can be defined over any algebraically closed field $k$. The construction is indicated schematically in the case where $T$ is non-degenerate in Figure \ref{fig1} and Equation \eqref{Eqn:PotentialForNondegenQuiver}, and defined more generally in \cite{LF1}. There is  a  triangulated \CY category $\D(T)$ of finite type (i.e. cohomologically finite) over $k$ defined by the Ginzburg algebra construction \cite{Ginzburg}; this has a distinguished heart, equivalent to the category of finite-dimensional modules for the complete Jacobi algebra of the quiver with potential $(Q(T),W(T))$.  Results of Keller-Yang and Labardini-Fragoso imply that $\D(T)$  depends up to quasi-isomorphism only on the underlying marked bordered surface $(\S,\M)$. We write $\D(\S,\M)$ for any category in this equivalence class.

 A meromorphic quadratic differential $\phi$ on a Riemann surface $S$ has an associated marked bordered surface $(\S,\M)$.  The surface $\S$ is obtained as the real blow-up of $S$ at poles of $\phi$ of order $\geq 3$; the distinguished tangent directions of the horizontal foliation of $\phi$ define boundary marked points in $\M$, and poles of order $\leq 2$ define the punctures (interior marked points) $\PP \subset \M$.  We will write $\Pol(\phi)$ for the set of poles, and $\Pol_{\geq i}(\phi)$, $\Pol_{=2}(\phi)$ for poles of constrained or specified orders.  For $q\in \Pol(\phi)$ let $ord(q)$ denote the order of the corresponding pole. 
 
 All quadratic differentials considered in this paper have simple zeroes.

 Let $\Quad(\S,\M)$ denote the complex orbifold parametrizing equivalence classes of pairs comprising a Riemann surface $S$ and a meromorphic quadratic differential $\phi$ with simple zeroes whose associated marked bordered surface is diffeomorphic to $(\S,\M)$. This has an open dense subset $\Quad(\S,\M)_0$ of pairs where the differential $\phi$ has poles of order exactly $2$ at $\PP$, or equivalently for which the flat metric defined by $\phi$ is complete.  There is an unramified $2^{|\PP|}:1$ cover $\Quad^{\pm}(\S,\M)_0 \rightarrow \Quad(\S,\M)_0$ whose points are signed meromorphic differentials, meaning that we fix a choice of sign of the residue of the differential $\phi$ at each double pole; the cover extends as a ramified cover $\Quad^{\pm}(\S,\M)\rightarrow \Quad(\S,\M)$.

\subsection{The threefolds}   

Fix a signed complete  differential $(S,\phi) \in \Quad^{\pm}(\S,\M)_0$.   Denote by $M$ the divisor $\sum_{p\in \mathrm{Pol}(\phi)} \lceil ord(p)/2 \rceil p$ on $S$, where $\lceil \alpha \rceil$ denotes the smallest integer greater than or equal to $\alpha$, and write
\begin{equation} \label{Eqn:divisorsplits}
M = M_{=2} + M_{\geq 3}
\end{equation}
corresponding to the obvious decomposition $\Pol(\phi) = \Pol_{=2} \amalg \Pol_{\geq 3}$.  Take a rank two holomorphic vector bundle $V$ on $ S$ with $\det(V) = K_S(M)$. We perform an elementary modification of the vector bundle $\Sym^2(V)$ at each double pole $p$ of $\phi$, to obtain a rank three  bundle $W$  fitting into a short exact sequence
\[0\lra W\lra \Sym^2(V)\lra \OO_{M_{=2}}\lra 0.\] 
We then remove from $S$ all poles of $\phi$ of order $\geq 3$.

The determinant map $\det\colon \Sym^2(V) \lra K_S(M)^{\otimes 2}$ restricts to a quadratic map $\det_W$ on the bundle $W$ which is rank one at points of $M_{=2} \subset S$.  The threefold
\[
Y_\phi = \big\{{\det}_W - \pi^*\phi = 0 \big\} \subset W|_{S-\Pol_{\geq 3}(\phi)}
\]
is an affine conic fibration over $S- \Pol_{\geq 3}(\phi)$, with nodal fibres over the zeroes of $\phi$,  fibres singular at infinity over the double poles, and empty fibres over higher order poles.  The topology of the fibre over a double pole $p$ depends on the choice of elementary modification. A choice of line in the fibre of $V$ at $p$ determines a distinguished elementary modification, with the property that the resulting fibre of $Y_{\phi}$ at $p$ is isomorphic to the disjoint union of two planes $\bC^2_{p,+} \amalg \bC^2_{p,-}$.   We will always consider elementary modifications with this property.

\subsection{The result}

We fix a linear K\"ahler form on $\bP(W \oplus \OO)$. This induces a K\"ahler form $\omega$ on $Y_{\phi}$.  A Moser-type argument, see Lemma \ref{Lem:WellDefined},  shows that (having fixed the  parameters determining the cohomology class of the K\"ahler form appropriately) the  symplectic manifold underlying $Y_\phi$  depends up to isomorphism only on the pair $(\S,\M)$.  $Y_{\phi}$ has vanishing first Chern class, and $\phi$ determines a distinguished homotopy class $\eta(\phi)$ of trivialisation of the  canonical bundle of $Y_\phi$.    

When $\partial \S \neq \emptyset$, the K\"ahler form $\omega$ is exact, and $Y_{\phi}$ has a well-defined exact Fukaya category $\scrF(Y_{\phi})$, which may be constructed over any field $k$, see \cite{FCPLT}.  More generally, 
we will consider Lagrangian submanifolds $L \subset Y_{\phi}$ with the following property: there is an almost complex structure $J_L$ on $Y_{\phi}$, taming the symplectic form and co-inciding with the given integrable structure at infinity, for which $L$ bounds no $J_L$-holomorphic disk and does not meet any $J_L$-holomorphic sphere.  For ease of notation, we will refer to such $(L, J_L)$ as \emph{strictly unobstructed}.  When $\S$ is closed, $Y_{\phi}$ has a strictly unobstructed Fukaya category (which we again denote by) $\scrF(Y_{\phi})$, now  defined over the Novikov field $\Lambda_{\bC}$.   This version of the Fukaya category appears, for instance, in \cite{AbouzaidSmith} and \cite{Seidel:Flux}.  The strictly unobstructed hypothesis rules out bubbling of holomorphic disks, which simplifies the technical construction of $\scrF(Y_{\phi})$, see \cite[Sections 3b, 3c]{Seidel:Flux} for a detailed discussion under slightly weaker hypotheses.  

Let $Y$ be a symplectic manifold with well-defined Fukaya category $\scrF(Y)$.  For each $b \in H^2(Y;\bZ_2)$, there is a category $\scrF(Y;b)$, the $b$-twisted strictly unobstructed Fukaya category, which for $b=0$ recovers the category considered previously.   Objects of $\scrF(Y;b)$ are closed oriented graded strictly unobstructed Lagrangians $L$, which are equipped with 
a relative spin structure\footnote{In other words, $b|_L = w_2(TL)$ and we fix a trivialisation of $\xi_B \oplus TL$ over the 2-skeleton of $L$, where $\xi_b\rightarrow Y$ is the unique real 2-plane bundle with $w_1(\xi_b)=0$ and $w_2(\xi_b)=b$.}, relative to the background class $b$. 
The choice of background class $b \in H^2(Y;\bZ_2)$ serves to change the signs with which holomorphic polygons contribute to the $A_{\infty}$-operations $\mu^d_{\scrF}$, and can be seen as fixing a particular coherent orientation scheme for the Fukaya category; compare to \cite[Section 11 \& Remark 12.1]{FCPLT}.

Each category $\scrF(Y_\phi;b)$ is a $\bZ$-graded $A_{\infty}$-category, linear over the appropriate field $k$.  Let $\D\,\scrC$ denote the  cohomological category $H^0(\mathrm{Tw}\,\scrC)$ of the category of twisted complexes over an $A_{\infty}$-category $\scrC$.

Now consider the threefold $Y_{\phi} \rightarrow S$. For each $p\in M \subset S$, the fibre $\pi^{-1}(p)$ is reducible. Let $\bC^2_{+,p}$ denote one component of this fibre. We fix the background class $b_0 \in H^2(Y_\phi;\bZ_2)$ represented  by the locally finite cycle 
\begin{equation} \label{Eqn:BackgroundCycle}
b_0 = \sum_{p \in M_{=2}} \bC^2_{p,+} \ \in H_4^{\operatorname{lf}}(Y_\phi;\bZ_2) \cong H^2(Y_\phi;\bZ_2)
\end{equation}
given by (either) one of the components of the reducible fibre at each point of $\Pol_{=2}(\phi)$, or equivalently each point of $\PP$.  The class is non-trivial by Lemma \ref{Lem:DivisorsIndpt}. Different choices of cycle representative for $b_0$ are related by monodromy by Lemma \ref{Lem:MonodromySwaps}.

\begin{Theorem} \label{Thm:Main2}
Let $(\S,\M)$ be a marked bordered surface, with $\M \neq \emptyset$.    Suppose either  
\begin{enumerate}
\item $\S$ is closed, $g(\S) > 0$, $|\M| \geq 3$, and $k=\Lambda_{\bC}$;   or 
\item $\partial \S \neq \emptyset$, and $\S$ is not a sphere with fewer than five punctures. 
\end{enumerate}
There is a $k$-linear fully faithful embedding  $\D(\S,\M) \hookrightarrow \D\scrF(Y_{\phi};b_0).$
\end{Theorem} 

The untwisted Fukaya category $\scrF(Y_{\phi})$, which differs from $\scrF(Y_{\phi};b_0)$ by certain signs, is different, and is discussed in Section \ref{Sec:Untwisted}.

One drawback of Theorem \ref{Thm:Main2} is that it does not give a symplectic characterisation of the image of the embedding $\D(\S,\M) \hookrightarrow \D\scrF(Y_{\phi};b_0)$.  There is an obvious candidate for such a characterisation, which we now explain.

Any embedded path $\gamma: [0,1] \rightarrow S$, with end-points distinct zeroes of $\phi$ and otherwise disjoint from the zeroes and poles of $\phi$, defines a Lagrangian 3-sphere $L_{\gamma} \subset Y_\phi$, fibred over the arc $\gamma \subset \Sigma$ via Donaldson's ``matching cycle" construction \cite[III, Section 16g]{FCPLT}.  The matching spheres  $L_{\gamma}$ are exact if $\partial \S \neq \emptyset,$ and strictly unobstructed (with $J_L$ the canonical complex structure on $Y_{\phi}$) when $\S$ is closed and $g(\S) > 0$.     A  Lagrangian sphere is relatively spin for any choice of background class $b$,  hence equipped with a grading defines a Lagrangian brane in $\scrF(Y_{\phi};b)$.  A non-degenerate ideal triangulation $T$ of $S$ defines a full subcategory $\scrA(T;b) \subset \scrF(Y_{\phi};b)$, generated by the matching spheres associated to the edges of the cellulation dual to $T$.   Theorem \ref{Thm:Main2}  is proved by showing that $\D\scrA(T;b_0) \simeq \D(T)$ for particularly well-behaved triangulations $T$.  Since the category $\D(T)$ does not depend on $T$,  it follows that $\D\scrA(T;b_0)$ also depends only on the pair $(\S,\M)$, up to triangulated equivalence.

 Let $\scrK(Y_\phi;b_0)\subset \scrF(Y_\phi;b_0)$ be the full $A_{\infty}$-subcategory  generated by Lagrangian matching spheres.  This manifestly depends only on the pair $(\S,\M)$.  It seems likely that $\D(\S,\M) \simeq \D\scrK(Y_{\phi};b_0)$. We outline one tentative approach to proving that, which amounts to proving that $\scrA(T;b_0)$ generates $\scrK(Y_{\phi};b_0)$, in Section \ref{Sec:Generation}, but elaborating the details of that sketch would be a substantial task.    Going further, it seems likely that all Lagrangian spheres in $Y_{\phi}$ are quasi-isomorphic to matching spheres (the result is proved in special cases in \cite{AbouzaidSmith, Seidel:C*equivt}), in which case the category $\D\scrK(Y_{\phi};b_0)$  would be a symplectic invariant of $Y_\phi$, carrying an action of the subgroup of $\pi_0\Symp(Y_{\phi})$ preserving the class $b_0 \in H^2(Y_{\phi};\bZ_2)$.  The question of whether the embedding $\scrK(Y_{\phi};b_0) \hookrightarrow \scrF(Y_{\phi};b_0)$ itself split-generates is also open, though here there seems to be less evidence either way.  We hope to return to these questions elsewhere.

The imposed condition $\M \neq \emptyset$ gives a substantial simplication since holomorphic polygons are constrained for grading considerations that do not pertain when $\phi$ is globally holomorphic, cf. Remarks \ref{Rem:NoTrivalentChance} and \ref{Rem:GradingMiracle}.   The constraint on the number of punctures $|\M| \geq 3$ when $\S$ is closed arises from a similar constraint in work of Geiss, Labardini-Fragoso and Schr\"oer \cite{GLFS}, who study the action of right equivalences on potentials on the quivers $Q(T)$, cf. Theorem \ref{Thm:GLFS1}.  We conjecture that, for a closed surface $\S$ of genus $g(\S)>0$, Theorem \ref{Thm:Main2} holds under the weaker hypothesis $|\M| > 1$.  (Once-punctured surfaces are special: not every pair of signed ideal triangulations are related by pops and flips, and when $g=1$ the analogue of Theorem \ref{Thm:GLFS1} is actually false, cf. \cite{GLFS}.) 

 One can relax the strict unobstructedness hypothesis at the cost of invoking the deep obstruction theory of \cite{FO3} in the construction of $\scrF(Y_{\phi};b)$. If $\S$ is closed and $g(\S) = 0$ then $Y_\phi$ may contain rational curves, so only the more complicated construction is available. This is the reason for the genus constraint in the first part of Theorem \ref{Thm:Main2}.   

\begin{Remark}
The elementary modifications appearing in the specific construction of $Y_{\phi}$ play a definite role in reproducing $\D(\S,\M)$, which seems somewhat less natural from the viewpoint of the symplectic topology of the original  bordered surface $(\S,\M)$, not least because of the \CY property.  The particular choice of elementary modification that we employ was motivated by two considerations: first, to yield the Calabi-Yau property of Lemma \ref{Lem:FirstChernZero}, and second, to ensure the non-vanishing of a certain count of local holomorphic sections over discs centred on double poles in Lemma \ref{Lem:LocalNonzero}.  (The latter result would fail if instead one took the threefold $Y_{\phi}$ to have smooth, nodal, higher multiplicity or empty fibres over the double poles, and also accounts for the appearance of the twisting class $b_0$.)  At a more technical level, the appearance of a reducible fibre whose components are exchanged by the local monodromy of the family of threefolds obtained by allowing the residue at a double pole to wind once around the origin, cf. Lemma \ref{Lem:MonodromySwaps}, fits well with the appearance of ``signed" quadratic differentials in \cite[Section 6.2]{BrSm}.
\end{Remark}

\subsection{Context}

In many cases, the paper \cite{BrSm} computes the space of stability conditions on the category $\D(\S,\M)$ in terms of moduli spaces of quadratic differentials. If either $\S$ is closed with $g(\S)>0$ and with at least two punctures, or $\partial \S \neq \emptyset$ and $\S$ is not a sphere with fewer than six punctures, there is a connected component $\Stab_{\Delta}(\S,\M)$ of $\Stab(\D(\S,\M))$ and a subgroup $\Aut_{\Delta}$ of autoequivalences which preserve this component modulo those which act trivially upon it, with
\[
\Stab_{\Delta}(\S,\M) / \Aut_{\Delta} \ \cong \ \Quad_{\heartsuit}(\S,\M)
\]
where $\Quad_{\heartsuit}(\S,\M)$ has the same coarse moduli space as $\Quad(\S,\M)$ but additional orbifolding along the incomplete locus, see \cite{BrSm} for details. This gives a non-trivial computation of the space of stability conditions on (a subcategory of) the Fukaya category of a symplectic six-manifold, and enables one to understand the Donaldson-Thomas invariants of these categories.  

One can construct a complex (3,0)-form $\Omega$ on $Y_\phi$ with the following property: if the path $\gamma \subset \Sigma$ is a saddle connection for $\phi$, the  associated matching sphere $L_{\gamma}$ is \emph{special}, i.e. $\Omega$  has constant phase on $L_{\gamma}$. (We are not working with the Ricci-flat metric, so these are not strictly  special in the traditional sense.)  There are similarly $\Omega$-special Lagrangian submanifolds $S^1\times S^2$ of $Y_\phi$ associated to homotopically non-trivial closed geodesics for the flat metric defined by $\phi$. Theorem 1.4 of \cite{BrSm} implies that the Donaldson-Thomas invariants of $Y_{\phi}$, defined with respect to the stability condition associated to $\phi$, count such special Lagrangian submanifolds, the existence and numerics of which are therefore governed by the Joyce-Song and Kontsevich-Soibelman wall-crossing formulae \cite{JS,KS}. This confirms, in this special case, a long-standing expectation of Joyce and others.  

It is natural to conjecture that moduli spaces of \emph{holomorphic} quadratic differentials with simple zeroes arise as spaces of stability conditions (modulo autoequivalences) on the Fukaya categories of the local threefolds of \cite{DDDHP} (see Equation \ref{Eqn:Holo3fold}). However, these categories do not appear to admit descriptions in terms of quivers, and different techniques would be required to analyse them and the corresponding spaces of stability conditions.  

\subsection{Higher rank}

The threefolds $Y_{\phi}$ are associated to meromorphic maps of Riemann surfaces into the versal deformation space of the $A_1$-surface singularity $\bC^2/\bZ_2$.  There are also local threefolds associated to meromorphic maps to the versal deformation spaces of other ADE singularities: in the case of a holomorphic map, the relevant threefolds are studied in \cite{DDDHP, Szendroi}.  Already for the $A_2$-surface, however, the geometry is substantially more complicated, and the relationship to stability conditions rather less clear.

The natural data required to write down a quasi-projective Calabi-Yau threefold fibred over a Riemann surface $S$ in the $A_n$-case is a tuple comprising (perhaps meromorphic) sections of $K_{S}^{\otimes i}$ for $2 \leq i \leq n$. For instance, when $n=2$ and supposing one is working with holomorphic rather than meromorphic differentials, one takes a vector bundle $V = L_1 \oplus L_2$ with $det(V) \cong K_S$, and considers the hypersurface
\begin{equation} \label{Eqn:HigherRank}
\left\{(x,y,z) \in L_1^3 \oplus L_2^3 \oplus L_1L_2 \ | \ xy + z^3 = \phi\cdot z + \psi \right\}
\end{equation}
where $\phi \in H^0(K_S^{\otimes 2})$ and $\psi \in H^0(K_S^{\otimes 3})$. For generic choices of $(\phi, \psi)$ this is a smooth Calabi-Yau, the total space of a Lefschetz fibration over $S$, with generic fibre the Milnor fibre of $\bC^2/\bZ_3$ (symplectically, the plumbing of two copies of $T^*S^2$). 

However, the moduli space of such data -- a complex structure on $\S$ and a tuple of differentials -- has dimension smaller than the dimension of the space of stability conditions on the corresponding category, or more mundanely smaller than the rank of $H_3$ of the associated threefold (compare to Remark \ref{Rem:H3}).  This is a familiar problem: whilst one expects spaces of complex structures on a symplectic Calabi-Yau to embed into the space of stability conditions on the Fukaya category, there is no reason to expect that embedding to be onto an open subset.  
 The fortunate accident in the $A_1$-case is that Teichm\"uller space has the same dimension as the space of quadratic differentials.  At least for $A_n$-fibred threefolds, it seems natural following \cite{Labourie} to expect the ``higher Teichm\"uller space", i.e. Hitchin's contractible component \cite{Hitchin} of the variety of flat $\bP SL(n,\bR)$-connections, to play a role in resolving this discrepancy.  

 Finally, we note that computations in \cite{GLMMN} indicate that the  \CY-categories arising in higher rank (even when one restricts to meromorphic differentials) are appreciably more complicated; whereas the Donaldson-Thomas invariants of the categories studied in this paper are always $+1$ or $-2$, cf. \cite{BrSm}, and non-vanishing only for primitive classes, in the higher rank case one expects there are classes $[\gamma]$ for which the DT-invariants (conjecturally related to counts of special Lagrangian submanifolds in the corresponding threefold) in class $k[\gamma]$ grow exponentially with $k$.   The symplectic topology of \eqref{Eqn:HigherRank} is the subject of work in progress by the author.

\subsection{Standing assumptions} The arguments for the two cases of Theorem \ref{Thm:Main2} are rather similar. For definiteness, for the rest of the paper we consider the (more complicated) case when $\S$ is closed, leaving the required modifications for the case with non-empty boundary to the interested reader.   

In the closed case all the marked points $\M=\PP \subset \S$ are punctures. To avoid transversality issues arising from rational curves and their multiple covers, we also exclude the (interesting) case of threefolds fibring over the 2-sphere.  Therefore the discrete topological data will henceforth be indexed by a pair $(\S,\M)$ with $g(\S) > 0$ and  $\# \M > 0$. 

\vspace{0.5em}

\noindent \Acknowledgements     Tom Bridgeland was originally to be a co-author; his influence is pervasive. I am indebted to Daniel Labardini-Fragoso for explanations of his joint work with Geiss and Schr\"oer \cite{GLFS}, on which we rely essentially. 
Thanks  also to Mohammed Abouzaid, Bernhard Keller, Yankl Lekili, Andy Neitzke, Oscar Randal-Williams, Tony Scholl, Paul Seidel, Balazs Szendroi and Richard Thomas for helpful conversations; Abouzaid and Seidel provided assistance with a crucial sign computation.  Finally, I am grateful to the anonymous referee for  suggesting many improvements to the exposition.  


\section{Background}

\subsection{Quivers with potential} \label{Sec:QuiverPotential}

 Let $Q$ be a quiver, specified by sets of vertices and arrows $Q_0$, $Q_1$, and source and target maps $s,t\colon Q_1\to Q_0$.  We write $kQ$ for the path algebra of $Q$ over the field $k$, and $\widehat{kQ}$ for the completion of  $kQ$ with respect to path length. A potential on $Q$ is an element $W\in \widehat{kQ}$ of the closure of the  subspace of $\widehat{kQ}$ spanned by all cyclic paths in $Q$ of length $\geq 2$. A  potential  is called reduced if it lies in the closure of the subspace  spanned by cycles of length $\geq 3$.

We say that two potentials $W$ and $W'$ are \emph{cyclically equivalent} if $W-W'$ lies in the closure of the subspace generated by differences $a_1\ldots a_s - a_2\ldots a_sa_1$, where $a_1\ldots a_s$ is a cycle in the path algebra.  $W$ and $W'$ are \emph{right-equivalent} if there is an automorphism $\phi: \widehat{kQ} \rightarrow \widehat{kQ}$ of the completed path algebra which fixes the zero-length paths and such that $\phi(W)$ and $W'$ are cyclically equivalent.   Following \cite{GLFS}, we say $W$ and $W'$ are \emph{weakly right equivalent} if $W$ and $t W'$ are right-equivalent, for some invertible scalar $t \in k^*$.

Consider minimal $A_\infty$  categories $\CC$ whose  objects $S_i$  are indexed by the vertices of $Q$, and such that, as a graded vector space
\[\Hom_{\CC}(S_i,S_j)=k^{\delta_{ij}}\oplus V^*_{ij}[-1]\oplus V_{ji}[-2]\oplus k^{\delta_{ij}}[-3],\]
where $V_{ij}$ is the space with basis consisting of arrows in $Q$ connecting vertex $i$ to vertex $j$.
There is an obvious non-degenerate pairing
\[\langle\blank,\blank\rangle\colon \Hom_\CC(S_i,S_j)\times \Hom_\CC(S_j,S_i)\to k[-3].\]
Thus, if we define \[c_{n}(f_n,\cdots, f_{1})= \langle f_n, m_{n-1}(f_{n-1},\cdots,f_{1})\rangle,\]
an $A_\infty$ product of degree $2-(n-1)$,
\[
m_{n-1}\colon \Hom_\CC(S_{j_{n-1}},S_{j_{n}})\tensor \cdots \tensor \Hom_\CC(S_{j_{1}},S_{j_{2}})\lra \Hom_\CC(S_{j_1},S_{j_{n}}), 
\]
is equivalently described by a linear map of degree $-n$
\begin{equation} \label{Eqn:CyclicVersion}
c_{n}\colon \Hom_\CC(S_{j_{n}},S_{j_{1}})\tensor \Hom_\CC(S_{j_{n-1}},S_{j_{n}}) \cdots \tensor \Hom_\CC(S_{j_{1}},S_{j_{2}})\lra k.
\end{equation}
Let us insist that $\CC$ is cyclic as an $A_\infty$ category, meaning that the tensors $c_n$ are cyclically invariant in the graded sense. If we further insist that the $A_\infty$ structure on $\CC$ is strictly unital then the whole structure is  determined by the elements $c_n(f_n,\cdots,f_1)$ when all the $f_i$ have degree 1.  For background on this construction, see \cite{Segal:DefPt}.

Let $W$ be a reduced potential on $Q$. Decomposing the potential into homogeneous pieces, and cyclically symmetrising, we obtain linear maps
\[W_{n}\colon V^*_{j_{n},j_1}\tensor V^*_{j_{n-1},j_{n}}  \tensor \cdots \tensor V^*_{j_{1},j_2} \lra k.\]
Setting $c_{n}=W_{n}$ gives a well-defined  $A_\infty$ category  $\CC(Q,W)$.
Define $\D(Q,W)$ to be the homotopy category of the category of twisted complexes over $\CC(Q,W)$
\[\D(Q,W)=H^0(\Tw(\CC(Q,W)).\]
The associated graded category  of $\D(Q,W)$ contains  $\CC$ as a full subcategory.

The same category $\D(Q,W)$ admits an alternative (Koszul dual) description in terms of the derived category of a dg algebra $A(Q,W)$ called the complete Ginzburg algebra (see for instance \cite[Theorem 9]{KS} or \cite[Section 5]{Keller:CY}).  To define it, first double $Q$, adding a dual edge $a^*$ for each $a\in Q$, and then add loops $t_i$ based at each vertex $i$ of $Q$.  The resulting quiver $Q^*$ has a grading given by
\[
\deg(x) = 0; \quad \deg(x^*) = -1; \quad \deg(t) = -2.
\]
Let $A(Q,W)$ be the completion of the path algebra $\widehat{kQ^*}$ as a graded algebra, with respect to the ideal generated by the arrows of $Q^*$.  There is a unique continuous differential $d$ satisfying
\[
d(t)= \sum_{a \in Q_1} e_i\cdot [a,a^*]\cdot e_i, \quad d(a^*)= \partial_a W, \quad d(a)= 0.
\]
Thus $A(Q,W)$ is a dg algebra.  
The category $\D(Q,W)$ can then be equivalently described as the full subcategory of the derived category of the dg algebra $A(Q,W)$ consisting of finite-dimensional modules.  By a general result of Keller and Van den Bergh \cite{Keller:Deformed}, this description shows that $\D(Q,W)$ is a \CY triangulated category. 

Keller and Yang \cite[Lemma 2.9]{KY} showed that  if $W$ and $W'$ are right-equivalent potentials, they have isomorphic complete Ginzburg algebras, and hence yield equivalent categories $\D(Q,W) \simeq \D(Q,W')$.  Ladkani \cite[Proposition 2.7]{Ladkani}, see also \cite[Lemma 8.5]{GLFS}, showed that the same conclusion holds when $W$ and $W'$ are only weakly right-equivalent. Indeed, there is a natural $k^*$-action on the set of minimal $A_{\infty}$-structures on the category $\CC$, where $\lambda \in k^*$ acts by rescaling the operation $m_n$ by $\lambda^{n-2}$.    $A_{\infty}$-structures related by the $k^*$-action are $A_{\infty}$-equivalent even though not gauge-equivalent in the usual sense (the required equivalence does not act by the identity on cohomology but by a multiple of the Euler vector field).  The  potentials $W$ and $tW$ on $Q$ give rise to $A_{\infty}$-categories $\CC(Q,W)$ and $\CC(Q,tW)$ which are related by the $k^*$-action.  Since $A_{\infty}$-equivalences induce equivalences on categories of twisted complexes by \cite[Lemma 3.20]{FCPLT},  the category $\D(Q,W)$ depends only on the weak right equivalence class of $W$.

 The fact that $A(Q,W)$ is  concentrated in non-positive degrees implies \cite[Lemma 5.2]{KY} that $\D(Q,W)$ is equipped with a canonical bounded t-structure, whose heart $\A(Q,W)$  is  equivalent to the category of nilpotent representations of the completed Jacobian algebra
\[
J(Q, W)=H^0(A(Q,W))=\widehat{kQ}\big/ (\partial_a W: a\in Q_1).
\]  
In particular, $\A(Q,W)\subset \D(Q,W)$ is a finite-length heart.  Since it admits a bounded t-structure, the category $\D(Q,W)$ is split-closed, i.e. agrees with its own idempotent completion, see \cite{LeChen}.

\subsection{Quivers from triangulated surfaces} \label{Sec:QuiverTriangulated}

 Suppose again that $\S$ is a closed oriented surface of genus $g(\S) > 0$, now equipped with a non-empty  set of $d\geq 2$  marked points $\M\subset \S$.   

By a non-degenerate ideal triangulation of $(\S,\M)$ we mean  a triangulation  of $\S$ whose vertex set is precisely $\M$, and in which every vertex has valency at least $3$ (this implies that every triangle has three distinct edges). A \emph{signed triangulation} is a triangulation equipped with a function
\[
\epsilon: \M \rightarrow \{\pm 1\}.
\]
We can associate a quiver with potential $(Q(T),W(T, \epsilon))$ to a signed  non-degenerate triangulation $T$ as follows. 
\begin{figure}[ht]
\begin{center}
\includegraphics[scale=0.4]{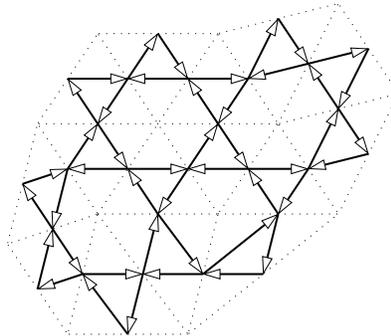}
\end{center}
\caption{Quiver associated to a triangulation\label{fig1}}
\end{figure}
The quiver $Q(T)$ has vertices at the midpoints of the edges of $T$, and is obtained by inscribing a small clockwise 3-cycle inside each face of $T$, as in Figure \ref{fig1}. There are two  obvious systems of cycles in $Q(T)$, namely a clockwise 3-cycle $T(f)$ in each face $f$, and an anticlockwise cycle $C(p)$ of length at least 3 encircling each point $p\in \M$. Define a potential
\begin{equation}\label{Eqn:PotentialForNondegenQuiver}
W(\T,\epsilon)=\sum_{f} T(f) - \sum_{p}  \epsilon(p) C(p).\end{equation}
When $\epsilon \equiv 1$ we will sometimes omit it from the notation.

Consider the derived category of  the  completed Ginzburg  algebra of the quiver with potential $(Q(T),W(T,\epsilon))$ over  $k$, and let $\D(T,\epsilon)$ be the full subcategory consisting of modules with finite-dimensional cohomology. As a special case of the discussion of Section \ref{Sec:QuiverPotential}, this is a \CY triangulated category of finite type over $k$, and comes equipped with a canonical t-structure, whose  heart $\A(T,\epsilon)\subset \D(T,\epsilon)$   is equivalent to the category of finite-dimensional modules for the completed Jacobi algebra of  $(Q(T),W(T,\epsilon))$.

Suppose two non-degenerate ideal triangulations $T_i$ are related by a flip,  in which the diagonal of a quadilateral is replaced by its opposite diagonal. The  resulting quivers with potential $(Q(T_i),W(T_i, \epsilon))$ (in which the signing is unchanged) are related by a mutation at the vertex corresponding to the edge being flipped. It follows from general results of Keller and Yang \cite{KY} that there exist distinguished $k$-linear triangulated equivalences $\Phi_\pm\colon \D(T_1, \epsilon)\isom \D(\T_2,\epsilon)$. 

Labardini-Fragoso \cite{LF1} extended the above definitions so as  to encompass a larger class of signed ideal triangulations (ones containing self-folded triangles, in which two of the three edges co-incide; in this case there may be punctures of valency one, and the mutation operation can change the signing).  He moreover  proved that flips also  induce right-equivalences  in this more general context.  There is another operation on signed ideal triangulations, which involves changing the signing at a given puncture without changing the triangulation, and a corresponding ``pop" equivalence which relates the associated categories. Under our hypothesis on $(\S,\M)$ that $d = |\M| \geq 2$, any two of these more general signed ideal triangulations are related by a finite chain of flips and pops. It follows  that up to $k$-linear triangulated equivalence, the category $\D(T,\epsilon)$ is independent of the chosen triangulation and of the choice of signing; see \cite[Sections 8 \& 9]{BrSm} for a more detailed discussion.  We denote by  $\D(\S,\M)$ any triangulated category in this quasi-equivalence class. 

Given the quiver $Q(T)$, one can define a \CY category by taking any potential on $Q(T)$, not necessarily the potential $W(T,\epsilon)$ described above. We will say that two potentials $W_1$ and $W_2$ are \emph{disjoint} if no cycle occuring in $W_1$ is cyclically equivalent to a cycle appearing in $W_2$.  The following result is due to Geiss, Labardini-Fragoso and Schr\"oer \cite{GLFS}. 

\begin{Theorem} \label{Thm:GLFS1}
Let $T$ be a triangulation of $(\S,\M)$ which contains no self-folded triangles or loops and in which every vertex has valency at least 4. Suppose that the associated quiver $Q(T)$ contains no double arrows.  Any two potentials on $Q(T)$ of the form
\begin{equation} \label{Eqn:General}
\sum_f T(f) - \sum_p \lambda_p C(p) + W'
\end{equation}
(for scalars $\lambda_p \neq 0$ and  $W'$ disjoint from the $T(f)$ and $C(p)$) are weakly right equivalent.
\end{Theorem}

Geiss, Labardini-Fragoso and Schr\"oer furthermore prove that every pair $(\S,\M)$ with $g(\S)>0$ and with $|\M| \geq 3$ admits some triangulation $T$ which satisfies the hypotheses, i.e. which contains no self-folded triangle or loop,  in which every vertex has valency at least 4, and for which the associated quiver has no double arrow. (In the case when $\S$ has non-empty boundary, the same result holds without further hypotheses on the number of punctures.) The proof of Theorem \ref{Thm:GLFS1} involves a delicate, iterative construction of a suitable right-equivalence by hand, obtained as an infinite composition of equivalences which, to first approximation, increase the minimal length of any cycle appearing in the remainder term $W'$; the actual proof is more complicated than this suggests.

For our purposes, these results yield  a finite-determinacy theorem for $A_{\infty}$-structures on the total endomorphism algebra of the category $\CC$ of Section \ref{Sec:QuiverPotential} in the special case $(Q,W) = (Q(T), W(T,\epsilon))$.  Considering the description of the category $\D(Q,W)$ as a category of twisted complexes over an  $A_{\infty}$-algebra given in Section \ref{Sec:QuiverPotential},  Theorem   \ref{Thm:GLFS1} implies in particular that different choices of scalars $\{\lambda_p\}_{p \in \PP}$ for the potential \eqref{Eqn:General} yield equivalent categories $\D(Q,W)$, whilst $A_{\infty}$-products encoded by the ``remainder" term $W'$ can be gauged away.

\subsection{Quadratic differentials and flat metrics}

Let $(S,\phi)$ denote a pair comprising a Riemann surface $S$ and meromorphic quadratic differential $\phi$ with poles of order precisely 2 at the points of a divisor $M\subset S$ comprising $d$ reduced points, and with simple zeroes.  Thus, the marked bordered surface associated to $(S,\phi)$ is diffeomorphic to $(\S,\M)$.
  Let $\Zer_{\phi} \subset S$ denote the set of zeroes, so 
\[
| \Zer_{\phi}| \, = \, 4g(S)-4+2|M|.
\]  At a point of $S \backslash \{M \cup \Zer_{\phi}\}$ there is a distinguished local co-ordinate $z$ with respect to which
\[
\phi(z)\,  = \, dz \otimes dz.
\]
This local co-ordinate is uniquely defined up to changes $z \mapsto \pm z + constant$. At simple zeroes, respectively double poles, there is a canonical co-ordinate with respect to which
\begin{equation} \label{Eqn:LocalCoord}
\phi(z) \, = \ \begin{cases} z \,dz^{\otimes 2} \quad \mathrm{at \ simple \ zeroes}; \\
m \, \frac{dz^{\otimes 2}}{z^2} \quad \mathrm{at \ double \ poles, \ where} \ m\in \bC^*. \end{cases}
\end{equation}
We refer to the value $m$ in the second case as the \emph{residue} at the double pole.

The surface $S\backslash M$ inherits a flat metric $|\phi|^2$ with singularities;  at each  $p\in \Zer_{\phi}$, the metric has a cone angle of $3\pi$.  The length element of the metric is defined by
\[
\sqrt{\phi(w)} dw
\]
in an arbitrary local parameter $w$, so the length of a curve $\gamma \subset S\backslash M$ is given by
\[
|\gamma|_{\phi} \ = \ \int_{\gamma} \, |\phi(w)|^{1/2} \, |dw|.
\]
This is well-defined for curves passing through zeroes of $\phi$, but diverges to infinity for curves through double poles. The area of the flat surface
\[
\int_S \, |\phi(w)| \, dx\wedge dy
\]
is  infinite, since a neighbourhood of each point of $M$ is isometric to a semi-infinite flat cylinder of circumference $2\pi|m|$, with $m$ as in Equation \ref{Eqn:LocalCoord}.  

The differential $\phi \in H^0(K_S(M)^{\otimes 2})$
  defines a \emph{horizontal foliation} of $S\backslash M$, given by the lines along which $arg(\phi) = 0$. In the natural local co-ordinate,  the horizontal foliation is given by lines $\Im(z)=constant$.  The local trajectory structure at a zero shows the horizontal foliation is not transversely orientable. The natural $S^1$-action by rotation, $\phi \mapsto e^{i\theta}\phi$, does not change the underlying flat surface, but changes which in the circle of foliations defined by $\arg\, \phi(z) = constant$ is regarded as horizontal.

A \emph{saddle connection} is a finite length maximal horizontal trajectory.  Any such has both end-points at (not necessarily distinct) zeroes of $\phi$.  

\subsection{WKB triangulations}

Suppose the quadratic differential $\phi$ is complete and saddle-free, meaning that it has no finite-length maximal horizontal trajectory. It then defines a canonical isotopy class of triangulation of $S$ with vertices at $M$, 
called the WKB-triangulation, see \cite[Section 10]{BrSm}. There is a dual ``Lagrangian cellulation", with trivalent vertices the zeroes of $\phi$ and which has exactly one face for each point of $M$.   

In general, the WKB-triangulation may contain self-folded triangles.
Given a quadratic differential $\psi$ whose associated WKB-triangulation contains a self-folded triangle, there is an edge in the Lagrangian cellulation which goes from a zero to itself.
The quiver prescription of Labardini-Fragoso differs in this case \cite{LF1}. 
For simplicity we will restrict attention to the non-degenerate case:

\begin{Lemma} \label{Lem:GoodWKB}
For every $g>0$ and $d>0$ there is a complete saddle-free differential whose associated WKB-triangulation contains no self-folded triangles. If $d>3$ one can assume that the triangulation satisfies the further hypotheses of Theorem \ref{Thm:GLFS1}.  Every edge of the dual cellulation then has distinct end-points. \end{Lemma}

\begin{proof}
According to \cite[Corollary 3.9]{FST}, any ideal triangulation can be transformed via a sequence of flips to a non-degenerate triangulation (one containing no self-folded triangles), whilst \cite{GLFS} constructs non-degenerate triangulations satisfying the hypotheses of Theorem \ref{Thm:GLFS1} whenever $d>3$.  Any non-degenerate triangulation $T$ has an associated bipartite quadrilation $\tilde{T}$, whose vertices are the vertices of $T$ together with the mid-points of all faces of $T$, and which has three edges for each face of $T$, which join the vertices of that face to its centre, see Figure \ref{Fig:Quadrilation}.  Section 4.9 of \cite{BrSm} shows that any quadrilation of the marked surface $(\S,\M)$ may be realised as the ``horizontal strip decomposition" of a quadratic differential $\phi$, meaning that $\phi$ has  double poles at the vertices of $T$,  zeroes at the additional (necessarily trivalent) vertices of $\tilde{T}$, and that the edges of $\tilde{T}$ are exactly the trajectories of $\phi$ which contain a zero.  It follows that the non-degenerate triangulation obtained in \cite{GLFS} is realised as a WKB triangulation.  The final statement is an immediate consequence of non-degeneracy. \end{proof}

\begin{center}
\begin{figure}[ht]
\includegraphics[scale=0.3]{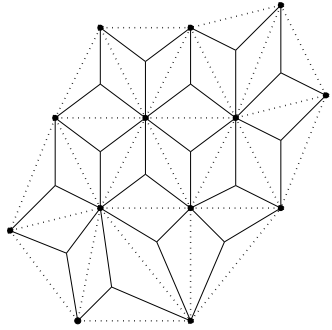}
\caption{The quadrilation associated to a triangulation\label{Fig:Quadrilation}}
\end{figure}
\end{center}

For brevity, 
we will call a triangulation as provided by Lemma \ref{Lem:GoodWKB} a non-degenerate WKB triangulation.

\section{Threefolds}\label{Sec:3folds}

We now associate quasi-projective Calabi-Yau 3-folds  to meromorphic quadratic differentials.  By way of motivation, if $(S,\psi)$ comprises a Riemann surface of genus $g\geq 2$ and a \emph{holomorphic} quadratic differential $\psi$ on $S$, there is a quasi-projective 3-fold which is a Lefschetz fibration over $S$, namely
\begin{equation} \label{Eqn:Holo3fold}
Y'_{\psi} \ = \ \left\{ (q_1, q_2, q_3) \in K_S^{\oplus 3} \, | \, q_1^2 + q_2^2 + q_3^2 = \psi\right\} \ \subset \mathrm{Tot}(K_S^{\oplus 3}).
\end{equation}
This is Calabi-Yau, and its geometry (Abel-Jacobi map, cycle theory etc) have elegant interpretations in terms of geometry on $S$, see \cite{DDDHP}.  Our spaces are cousins of these, adapted to the case of meromorphic differentials.

\subsection{Elementary modification} \label{Sec:TotalSpaces}

Let $S$ be a closed Riemann surface of genus $g\geq 1$, equipped with a reduced divisor $M$ comprising $d=|M|$ points.   Fix a rank two holomorphic vector bundle
\[
V \rightarrow S \qquad \textrm{with  an  isomorphism} \ \eta: \det(V) \cong K_S(M)
\]
The determinant defines a fibrewise quadratic map
\begin{equation} \label{Eqn:determinant}
\det: S^2(V) \rightarrow K_S(M)^{\otimes 2}.
\end{equation}
Consider an elementary modification of the symmetric square $S^2(V)$ along $M$, fitting into a short exact sequence of sheaves 
\begin{equation} \label{Eqn:ElemMod}
0 \rightarrow W \longrightarrow S^2V \stackrel{\alpha}{\longrightarrow} (\iota_M)_* \bC \rightarrow 0.
\end{equation}

\begin{Lemma} \label{Lem:ElemMod}
$W$ is locally free of rank $3$, and $c_1(W) = 3K_S +2\mathrm{PD}[M]$.
\end{Lemma}

\begin{proof}
The sheaf $W$ is torsion-free on a smooth curve, hence locally free.  The first Chern class is given by the Whitney sum formula.
\end{proof}

Elementary modifications along $M$ are not unique, but depend on the choice of $\alpha$ in \eqref{Eqn:ElemMod}.   It will be important for us to choose the elementary modification compatibly with the quadratic map \eqref{Eqn:determinant}.

\begin{Lemma} \label{Lem:ElemMod2}
There is an elementary modification $W$ as above with the property that for each $x\in M$, the induced map
\[
\det: W \rightarrow K_S(M)^{\otimes 2}
\]
has fibre $\bC^2 \amalg \bC^2 \subset W_x$ over any non-zero point of $K_S(M)^{\otimes 2}_x$.
\end{Lemma}

\begin{proof}
The statement is obviously local at a given $x \in M$.  Let $x$ be defined by an equation $f$. Near $x$ we fix a trivialisation  $S^2V \cong \calO \oplus \calO \oplus \calO$ in which $S^2V$ is spanned by holomorphic sections $a,b,c$ with respect to which the determinant map is given by the fibrewise quadratic
\[
(a,b,c) \mapsto ab-c^2.
\]  Such a local basis of sections arises naturally from a choice of local basis of sections $s, s'$ for $V$ near $x$, with $a = s\otimes s$, $b = s'\otimes s'$ and $c = s\otimes s' + s'\otimes s$.  The proof of Lemma \ref{Lem:ElemMod} implies there is an elementary modification $W$ which is spanned by local holomorphic sections $(fa,b,c)$, and the determinant map on $W$ is then given by 
\[
(fa)b - c^2.
\]
At $x$, where $f=0$, the fibre $\det^{-1}(y) = \{c^2 = -y\} \subset W_x$ is isomorphic to the disjoint union of two planes $\{c = \pm \sqrt{-y}\} \subset \bC^3$, provided $y\neq 0$.
\end{proof}

Note that a choice of complex line $\ell$ in the fibre $V_p$ of $V$ at $p$, equivalently of a parabolic structure on $V$ at $p$, induces an elementary modification as in Lemma \ref{Lem:ElemMod2}, where the subspace $W_p \subset S^2V_p$ is identified with the quadratic forms on $V_p^*$ vanishing on the annihilator of $\ell \subset V_p$.

\subsection{A quasiprojective Calabi-Yau 3-fold}

Let $\phi\in H^0(K_S(M)^{\otimes 2})$ be a meromorphic quadratic differential on $S$ with simple zeroes and a pole of order exactly $2$ at each $x\in M$. Define the hypersurface
\[ Y_\phi = \{det_{W} - \pi^*\phi = 0\} \subset W\] 
inside the total space of the vector bundle $W$.  We shall write
\[
X_{\phi} \subset \bP(W\oplus \calO)
\]
for the fibrewise projective completion of $Y_{\phi}$.  Being fibred in quadrics, this is the zero-locus of a section of $\pi^*K_S(M)^{\otimes 2}(2) \rightarrow \bP(W\oplus\calO)$, where $\pi: \bP(W\oplus\calO)\to S$ denotes projection.

The previous description of the determinant map $\det_W$ shows that the natural map $X_{\phi} \rightarrow S$ is a fibration by projective quadrics, with generic fibre $\bP^1\times \bP^1$, nodal fibres over zeroes of $\phi$, and reducible  fibre $\bP^2 \cup_l \bP^2$ at each point of $M$, i.e. the union of two planes joined along a line $l$.  The locus of reducible quadrics has codimension $>1$ in the space of all quadric hypersurfaces in $\bP^3$, so a generic one-parameter family of quadrics would have no such singular fibres.  The singular fibre $\bP^2 \cup_l \bP^2$ is not locally smoothable (i.e. not the $0$-fibre of a smooth three-fold $\scrX \rightarrow D$); an infinitesimal smoothing is determined by a section of the tensor product $\nu_{l/\bP^2}^* \otimes \nu_{l/\bP^2}^* \cong \mathcal{O}(-2)$ of the normal bundles to the normal crossing locus $l$ in the two components of the singular fibre.  

\begin{Lemma}
 $X_{\phi}$  has two isolated singularities at infinity (i.e. in the complement of $Y_{\phi} \subset X_{\phi}$) over each point of $M \subset S$, which are 3-fold ordinary double points.  These are the only singularities of $X_{\phi}$.
\end{Lemma}

\begin{proof}
Away from $M$, the map $X_{\phi}\rightarrow S$ is a Lefschetz fibration, and smoothness of the total space is clear. 
Given the description of the determinant map in Lemma \ref{Lem:ElemMod2}, a local model for the behaviour near the singular fibres over $M$ is given by a neighbourhood of the $(\delta=0)$ fibre in the quadric pencil 
\begin{equation} \label{Eqn:Model}
\left\{ x^2- \delta yz = \phi(\delta) t^2 \right\} \ \subset \ \bP^3_{[x:y:z:t]} \times \bC_{\delta}; \qquad \phi(\delta) \neq 0 \, \forall \, \delta
\end{equation}
The subspace $\{t=1\}$ defines the vector bundle $W \subset \bP(W\oplus\calO)$ in the given trivialisation, and by hypothesis  the holomorphic function $\phi(\delta)$ vanishes away from $M$.   Without loss of generality, we can suppose $\phi(0)=1$.  Under projection to the second factor $\bC_{\delta}$, the 0-fibre is $\{x^2=t^2\}$ which is a union of two planes, whose line of intersection $\{x=0=t\}$ lies in the hyperplane at infinity $\Pi = \{t=0\}$.  The complement of $\Pi$ is the affine variety
\[
\left\{\, x^2 - \delta yz = \phi(\delta) \, \right\} \ \subset \ \bC^4.
\]
Under the projection to the plane $\bC_{\delta}$, this has generic fibre an affine conic $T^*S^2$, and these degenerate at $\delta = 0$ to a singular fibre $\bC^2 \amalg \bC^2$.  The singularities of the total space of \eqref{Eqn:Model} are the points 
\[
y=1, \, 0 \in \{x^2-\phi(\delta)t^2 = \delta z\} \qquad \textrm{and} \ z=1,\,  0 \in \{x^2-\phi(\delta) t^2=\delta y\}
\]
in the given affine charts $\{y=1\}$ respectively $\{z=1\}$, which since $\phi$ is locally non-vanishing are both 3-fold ordinary double points.
\end{proof}

\begin{Corollary}
$Y_{\phi}$ is smooth.
\end{Corollary}

\begin{proof}
Removing the section of $\mathcal{O}_{\bP}(-1)$ defining the divisor $\Delta_{\infty} = X_{\phi} \backslash Y_{\phi}$ at infinity removes the line $l$ from each reducible fibre, hence removes all the nodes.
\end{proof}

The quasi-projective variety $Y_\phi$ comes with a natural projection map $\pi: Y_{\phi} \rightarrow S$. 
\begin{itemize}
\item The generic fibre of $\pi$ is a smooth affine quadric $\{ab-c^2 = t\}$ with $t\neq 0$,  abstractly diffeomorphic to the cotangent bundle $T^*S^2$;
\item At a zero $p$ of $\phi$, $\pi^{-1}(p)$ is defined by the quadratic $\{ab-c^2 = 0\} \subset \bC^3$, which has an isolated nodal singularity;
\item At a point $x\in M$, recalling that by hypothesis $\phi(x)=t' \neq 0$, the fibre $\pi^{-1}(x)$ is given by an affine quadric $\{c^2 = t'\} \subset \bC^3$, a disjoint union of two planes.
\end{itemize}

\begin{Lemma} \label{Lem:FirstChernZero}
$Y_{\phi}$ has holomorphically trivial canonical bundle, hence  $c_1(Y_{\phi})$ = 0.  The choice of isomorphism $\eta: \det(V) \rightarrow K_S(M)^{\otimes 2}$ defines a canonical homotopy class of trivialisation of the canonical bundle $K_{Y_{\phi}}$.
\end{Lemma}

\begin{proof}
Consider the  $\bP^3$-bundle $\pi: \bP = \bP(W \oplus \mathcal{O}) \rightarrow S$.  The determinant map
\[ det: S^2(V) \rightarrow K_{S}(M)^{\otimes 2} \]
restricted to $W \subset S^2(V)$ is fibrewise quadratic, hence can be viewed as an element of the space of global sections $\pi^*(K_{S}(M)^{\otimes 2})(2)$, which pushes forward to give 
\[
\Gamma\,\big( S^2(W \oplus \mathcal{O})^* \otimes  \pi^*K_{S}(M)^{\otimes 2}\big).
\]
The projective completion $X_{\phi}$  has canonical class $K_X = K_{\bP} K_{S}(M)^{\otimes 2} (2)$. Since
\begin{equation} \label{Eqn:CompletionCanonical}
K_{\bP} = det(W^*\oplus\mathcal{O})(-4)\otimes \pi^*K_{S}
\end{equation}
and there is an isomorphism $\tilde{\eta}: det(W^*) \rightarrow K_{S}(M)^{\otimes -3}(M)$, one sees 
\[
K_X = \mathcal{O}_{\bP}(-2).
\]
The quasi-projective subvariety $Y_{\phi} \subset X_{\phi}$ is  the complement of the section of $\mathcal{O}_{\bP}(1)$ at infinity, and the square of that section is a canonical divisor on $X_{\phi}$. The last statement follows from \eqref{Eqn:CompletionCanonical} since $\eta$ induces $\tilde{\eta}$. 
\end{proof}

\begin{Lemma} \label{Lem:Gradable}
There is a nowhere zero holomorphic volume form $\kappa_{\phi}$ on $Y_{\phi}$.
\end{Lemma}

\begin{proof}
Up to rescaling, there is a unique  section of $\calO_{\bP}(1)$ vanishing on the divisor $\Delta_{\infty}$ at infinity, and since $K_{X_{\phi}} = \calO_{\bP}(-2)$, the complement $Y_{\phi}$ admits a canonical holomorphic volume form up to scale.  
\end{proof}

\begin{Remark}
The form $\kappa_{\phi}$ has poles of order 2 at infinity.  For a heuristic discussion of the relevance of the pole order being $\geq 2$ to constructions of stability conditions on the Fukaya category starting from pairs comprising a complex structure and such a non-vanishing volume form, see  \cite[Section 7.3]{KontSoi:WC}.
\end{Remark}

\subsection{Resolution}

The 3-fold ordinary double point 
\[
\{z_0^2+z_1^2+z_2^2+z_3^2 = 0\} \subset \bC^4
\]
admits two distinct small resolutions, in which the singular point is replaced by a smooth $\bP^1$ with normal bundle $\calO(-1)\oplus \calO(-1)$.  The resolutions are obtained by collapsing either one of the two rulings of the exceptional divisor $\bP^1\times \bP^1$ resulting from blowing up the singularity; the passage from one resolution to the other is the simplest example of a 3-fold flop.  
 
\begin{Lemma} \label{Lem:SmallRes}
There is a projective small resolution $\chi: \hat{X}_{\phi} \rightarrow X_{\phi}$.
\end{Lemma}

\begin{proof}
Blow-ups of projective varieties are necessarily projective. Let $\hat{X}_{\phi}$ be given by blowing up a Weil divisor comprising exactly one component $\bP^2 \cong H_x \subset \pi^{-1}(x)$ of each reducible fibre, $x\in M$.   Such a Weil divisor contains all of the nodes, hence the blow-up is a small resolution.  See \cite{Clemens}, \cite{STY} for further discussion. 
\end{proof}

Let $\hat{\Delta}_{\infty} \subset \hat{X}_{\phi}$ be the total transform of $\Delta_{\infty} \subset X_{\phi}$.
 \begin{center}
 \begin{figure}[ht]
 \includegraphics[scale=0.5]{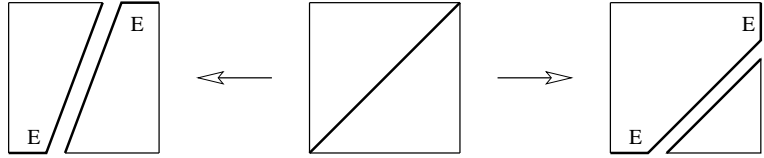}
 \caption{Reducible degenerations of a quadric surface with smooth total space. The divisor $\hat{\Delta}_{\infty}$ is marked in  bold (sum of components with multiplicity one); small resolution curves are labelled $E$. \label{Fig:ToricDegen}}
 \end{figure}
 \end{center}

\begin{Lemma} \label{Lem:SmoothInfinity}
The divisor $\hat{\Delta}_{\infty} \subset \hat{X}_{\phi}$ is smooth, and supports an effective anticanonical divisor.
\end{Lemma}
 
 \begin{proof}
 There are two possible models for the singular fibre of $\hat{X}_{\phi} \rightarrow S$ over a point of $M$, depending on whether the small resolution curves lie in the same or distinct components of the fibre. Either the generic fibre $\bP^1\times \bP^1$ degenerates to a union of two first Hirzebruch surfaces meeting along a fibre; or it degenerates to a copy of the blow-up of $\bP^2$ at two distinct points $p,q$, meeting a second copy of $\bP^2$ along the curve which is the proper transform of the line between $p,q$, see Figure \ref{Fig:ToricDegen}.  More explicitly, the local model given by the rightmost of the degenerations of Figure \ref{Fig:ToricDegen} can be obtained by taking the trivial fibration $(\bP^1\times\bP^1) \times D$ and blowing up a point $((q,q),0)$ lying on the diagonal of the central fibre.
The two models are related by flopping one of the $(-1,-1)$-curves $E$. In either case, $\hat{\Delta}_{\infty}$ is a conic bundle over $S$ with Lefschetz-type singularities, hence the total space of the divisor is smooth. Since $\hat{X}_{\phi} \rightarrow X_{\phi}$ is crepant and $\Delta_{\infty} \subset X_{\phi}$ supports an effective anticanonical divisor on $X_{\phi}$, by Lemma \ref{Lem:FirstChernZero}, the last statement holds.
 \end{proof}

\begin{Lemma} \label{Lem:NoCurve}
 $\hat{\Delta}_{\infty} \cdot C \geq 0$ for every rational curve $\bP^1 \cong C\subset \hat{\Delta}_{\infty} \subset \hat{X}_{\phi}$. 
\end{Lemma}

\begin{proof}
Inside $X_{\phi}$, the divisor at infinity is a conic bundle over $S$, with an integrable complex structure for which projection to $S$ is holomorphic.  Since $g(S) > 0$, any rational curve in $\Delta_{\infty}$ is contained in a fibre of projection, hence is a rational curve in some quadric surface in $\bP^3$.  Any such curve is in the homology class of some multiple of a line in $\bP^3$, hence deforms inside the $\bP^3$-fibre, and meets $\Delta_{\infty}$ strictly positively.

The small resolution $r: \hat{X}_\phi \rightarrow X_\phi$ contracts, for each point $x\in M$, two smooth $\bP^1$'s with normal bundle $\calO(-1)\oplus\calO(-1)$. For such a curve $E$, the intersection $E \cdot \hat{\Delta}_{\infty} = 0$; indeed, $K_{\hat{X}_{\phi}} = r^*K_{X_{\phi}} = r^*(-2 \Delta_{\infty}) = -2\hat\Delta_{\infty}$, but the canonical class is trivial near a $(-1,-1)$-curve. The result for a general $C=\bP^1 \subset \hat{\Delta}_{\infty}$ then follows by linearity; the coefficient of a line in the homology class of $C$ must be non-negative by considering the area (with respect to a suitable K\"ahler form, see Section \ref{Sec:SymplForms} below) of the image of $C$ after blowing down.
\end{proof}

At a double pole of $\phi$, there is a  canonical local complex co-ordinate on $S$ in which a quadratic differential can be expressed as
\[
\phi = m \, \frac{dz^{\otimes 2}}{z^2}.
\]
We refer to $m\in \bC^*$ as the \emph{residue} of $\phi$ at the double pole.

For each $p\in M \subset S$, the fibre $\pi^{-1}(p)$ is reducible. Let $\bC^2_{+,p}$ denote one component of this fibre. 

\begin{Lemma} \label{Lem:DivisorsIndpt}
The divisors $\bC^2_{+,p}$ are linearly independent in $H^2(Y_{\phi};\bZ_2) \cong H_4^{lf}(Y;\bZ_2)$.
\end{Lemma}

\begin{proof}
By considering intersections with the small resolution curves, the $d$ divisors defined by taking one component of each reducible fibre are linearly independent in $H^2(\hat{X}_{\phi};\bZ)$; general properties of small resolutions \cite{STY} further imply that $b_2(\hat{X}_{\phi}) = b_2(X_{\phi}) + d$.   $Y_{\phi}$ is the complement of a smooth divisor $\hat{\Delta}_{\infty} \subset \hat{X}_{\phi}$, by Lemma \ref{Lem:SmoothInfinity}. The complex surface $\hat{\Delta}_{\infty}$ is a ruled surface over $S$, with fibres comprising a chain of 3 rational curves over each point of $M$ and smooth fibres elsewhere. Mayer-Vietoris gives an exact sequence with $\bZ_2$-coefficients
\begin{equation} \label{Eqn-mv}
0 \rightarrow H^2(\hat{X}_{\phi}) \rightarrow H^2(Y_{\phi}) \oplus H^2(\hat{\Delta}_{\infty}) \rightarrow H^2(\partial U(\hat{\Delta}_{\infty})) \rightarrow 0
\end{equation}
with $\partial U(\hat{\Delta}_{\infty})$ the smooth five-manifold which is the circle normal bundle to $\hat{\Delta}_{\infty}\subset \hat{X}_{\phi}$.  The Gysin sequence for the cohomology of this five-manifold shows the map 
\[
H^2(\hat{\Delta}_{\infty}) \rightarrow H^2(\partial U(\hat{\Delta}_{\infty}))
\]
is a surjective map $\bZ_2^{2+2d} \rightarrow \bZ_2^{1+2d}$ with rank one kernel spanned by the Euler class. The group $H^2(\hat{X}_{\phi})$ has rank $2+d$,  and dimension counting shows that $H^2(Y_{\phi})$ has rank $1+d$ and that the map between them in \eqref{Eqn-mv} has full rank. It follows that the components of the reducible fibres in $H^2(\hat{X}_{\phi})$ are linearly independent in the image.
\end{proof}

\begin{Lemma} \label{Lem:MonodromySwaps}
Consider a loop $\gamma$ of quadratic differentials $\{\phi_t\}_{t\in S^1}$ on $S$ with the property that the residue at a given double pole $p \in S$ has winding number $+1$ about $0$. Let $f: \scrX\rightarrow S^1$ denote the corresponding family of relative quadrics, with fibre $X_{\phi_t}$ over $\phi_t$.  The monodromy of $f$ on $H_*(X_{\phi_1})$ exchanges the homology classes of the two components of the reducible singular fibre $\pi^{-1}(p) \subset X_{\phi_1}$.
\end{Lemma}

\begin{proof}
In the local model $\{x^2-\delta yz = \phi(\delta)t^2\}$ of Equation \eqref{Eqn:Model}, consider a family of differentials $\phi_{\theta}$ with $\phi_{\theta}(0) = e^{i\theta}$.  The components of the fibre over $\delta = 0$ of the affine piece $Y_{\phi_{\theta}}$ are given by $\{x = \pm e^{i\theta/2}t\}$, which are exchanged by the monodromy corresponding to varying  $\theta$ in $[0,2\pi]$.\end{proof}

Since the choice of small resolution $\hat{X}_{\phi}$ depends on a choice of component of the reducible fibre along which to blow up, there is no obvious universal family of small resolutions over any such loop $\gamma$ in the space of quadratic differentials. A universal family of small resolutions does exist over the space $\Quad^{\pm}(\S,\M)_0$ of signed complete differentials introduced in the Introduction.

\subsection{Topology}  We consider the algebraic topology of the threefold $Y_{\phi}$.

\begin{Lemma} \label{Lem:LocalTop}
If $\pi: Y_{\phi} \rightarrow S$ is the natural projection and $D_p \subset S$ is a small disk encircling a pole $p\in M = \Pol(\phi)$, then $\pi^{-1}(D_p) \subset Y_{\phi}$ is simply connected, and has reduced homology groups 
\[
H_2(\pi^{-1}(D_p)) \cong \bZ^2; \quad H_3(\pi^{-1}(D_p)) \cong \bZ.
\] 
\end{Lemma}

\begin{proof}
Via the right side of Figure \ref{Fig:ToricDegen}, a neighbourhood of the reducible fibre at a point of $M$ is described topologically as follows.  Let $u: \bP^1 \times \bP^1 \times \bP^1 \rightarrow \bP^1$ denote the third projection.  Let $H \subset (\bP^1)^3$ be a divisor which is a smooth conic in every fibre of $u$ over $\bP^1 \backslash \{0\}$, but meets the $0$-fibre of $u$ in a union of two lines.  Let   $Z$ denote the blow-up of $\bP^1 \times \bP^1 \times \bP^1$ at the unique intersection point of those two lines, and let $W \subset Z$ denote the divisor which is the proper transform of $u^{-1}(\infty) \cup H$.    Then $\pi^{-1}(D_p) \simeq Z \backslash W$, and hence $H_*(\pi^{-1}(D_p)) \cong H^{6-*}_{ct}(Z\backslash W) = H^{6-*} (Z,W)$. The computation is then straightforward.
 \end{proof}
 
The Lemma implies that $Y_{\phi}$ contains homotopically non-trivial 2-spheres which are not contained in a fibre of projection to $S$,  which is one source of delicacy in the subsequent construction of its Fukaya category.
 
\begin{Lemma} \label{Lem:H3}
$H_3(Y_{\phi}; \bZ)$ has rank $6g-6+3d$; the intersection form has kernel of rank $d$. 
\end{Lemma}

\begin{proof}
Let $Y_{\phi}^{op}$ denote $\pi^{-1}(S\backslash \cup_p D_p)$, with $D_p$ a small disk enclosing $p\in M$ and no other critical point of $\phi$. A Mayer-Vietoris argument and (the proof of) Lemma \ref{Lem:LocalTop} implies  $H_3(Y_{\phi}) \cong H_3(Y^{op}_{\phi})$.  We now apply the Leray-Serre spectral sequence to the projection $Y^{op}_{\phi} \rightarrow S\backslash M$. The monodromy in $R^2\pi_* \bZ$  of a projective fibration of quadric surfaces can be canonically identified with the monodromy in $R^0 \pi_* \bZ$ of the associated double covering of Riemann surfaces, see e.g. \cite{Reid}.  Let $C\rightarrow S$ be the double cover branched at the zeroes of $\phi$, and $C^0 \subset C$ the preimage in $C$ of $S\backslash M$. We next identify the 2-dimensional homology of an affine quadric with the anti-invariant 0-dimensional homology of the corresponding pair of points. Then
\[
H_3(Y^{op}_{\phi};\bZ) \ \cong \ H_1(S\backslash M; R^2 \pi_* \bZ) \ \cong \ H_1(C^0; \bZ)^-.
\]
The last group was computed by Riemann-Hurwitz in \cite[Lemma 2.2]{BrSm}, and has rank $6g-6+3d$.  Matching paths in $S$ between zeroes of $\phi$ define circles $\gamma \subset C^0$ and 3-spheres $L_{\gamma} \subset Y_{\phi}$, cf. Section \ref{Sec:LagSphere} below. By considering a basis of either group associated to matching paths of a cellulation of $S$, one sees that the intersection forms $\langle \bullet, \bullet \rangle_{C^0}$ and $\langle \bullet, \bullet \rangle_{Y_{\phi}}$ agree, which means that the kernel of the intersection form can be computed on $C^0$.  The last statement then follows from  \cite[Section 2]{BrSm}.
\end{proof}

\begin{Remark} \label{Rem:H3}
The space of stability conditions  $\Stab(\D\scrC)$ on any triangulated category $\scrC$  is locally homeomorphic to $\Hom_{\bZ}(K^0(\scrC),\bC)$. The $K^0$-group of the \CY category $\D(Q,W)$ defined by a quiver with potential is freely generated by the vertices of the quiver, and for the quivers $Q(T)$ arising from ideal triangulations of $(\S,\M)$, the number of vertices is $6g-6+3d$. On the other hand, for any symplectic manifold $Y^{2n}$ for which the Fukaya category is well-defined and $b\in H^2(Y;\bZ_2)$, there is always a natural homomorphism
\[
K^0(\scrF(Y;b)) \longrightarrow H_n(Y;\bZ)
\]
which associates to a Lagrangian submanifold its homology class (note we have not passed to split-closures).  From  Lemma \ref{Lem:H3} and Theorem \ref{Thm:Main2} one can show that this map is an isomorphism for $Y_{\phi}$ if one restricts to the $K^0$-group of the subcategory $\scrA(T;b_0) \subset \scrF(Y_{\phi};b_0)$ introduced after Theorem \ref{Thm:Main2}.  It is interesting to compare this to Abouzaid's computation \cite{Abouzaid:Ktheory}  for $K^0(\scrF(\Sigma_g))$, with $\Sigma_g$ a closed surface of genus $g\geq 2$.  
\end{Remark}

\subsection{Symplectic forms} \label{Sec:SymplForms}
The divisor $\Delta_{\infty} \subset X_{\phi}$ is relatively ample over $S$, and its pullback $\hat{\Delta}_{\infty} \subset \hat{X}_{\phi}$ is relatively nef, and relatively ample over $S\backslash M$.  Fix a Hermitian metric $\| \cdot \|$ in $\calO(\hat{\Delta}_{\infty})$ for which the curvature form $iF_{\nabla}$ is a semipositive $(1,1)$-form. Denote by $t$ the section of $\calO_{\hat{X}_{\phi}}(\hat{\Delta}_{\infty})$ defining the divisor at infinity. The 2-form
\begin{equation} \label{Eqn:VerticalForm}
\omega_v = -dd^c(log \|t\|^2)
\end{equation}
is weakly plurisubharmonic and vertically non-degenerate over $S\backslash M$; for each $x\in S\backslash (M \cup \Zer_{\phi})$ the fibre $(\pi^{-1}(x), \omega_v)$ is a finite type Stein manifold, symplectomorphic to a Stein subdomain of the cotangent bundle $(T^*S^2,dp\wedge dq)$.  In particular, the fibres of $Y_{\phi}$ over $S-M$ have contact type at infinity.

Fix an area form $\omega_S$ on $S$ of total area $d=|M|$.  The class $[\pi^{-1}(pt)] + \lambda \Delta_{\infty}$ lies in the interior of the ample cone of $X_{\phi}$ for any $\lambda > 0$, and the form 
\[
\lambda \omega_v +  \pi^* \omega_S
\] 
is symplectic away from the small resolution curves  $\{E_i\} \subset \hat{X}_{\phi}$.   
Recall that $\chi: \hat{X}_{\phi} \rightarrow X_{\phi}$ is the blow-up of a (not necessarily connected) Weil divisor $H_M \subset X_{\phi}$ which passes through all the ordinary double points. Flopping the small resolution curves appropriately, we can ensure that the pullback $H \subset \hat{X}_{\phi}$ of $H_M$ meets every $E_i$ strictly positively.  Direct consideration of the blow-up, cf. the proof of \cite[Theorem 2.9]{STY},  implies that 
\begin{equation} \label{Eqn:SymplecticForm}
\omega_{res} = \lambda \omega_v +  \pi^*\omega_S + \delta \sigma_H
\end{equation}
is a  K\"ahler form on $\hat{X}_{\phi}$, 
where $\sigma_H$ is a 2-form Poincar\'e dual to  $H_M$, pointwise positive on each of the $E_i$, and $\delta >0$ is sufficiently small. 

\begin{Lemma} \label{Lem:Parallel}
Let $\gamma: [0,1] \rightarrow S-(M \cup \Zer_{\phi})$ be a $C^1$-smooth embedded path. For any $\lambda > 0$, there is a well-defined symplectic  parallel transport map for $\pi: Y_{\phi} \rightarrow S$ over $\gamma$, which induces an exact symplectomorphism of the fibres $\pi^{-1}(\gamma(i))$, $i=0,1$ over the end-points.
\end{Lemma}

\begin{proof}
Since by hypothesis $\gamma$ avoids $M$, and the perturbing form $\sigma_H$ can be chosen to be supported near the preimage of $M$, it suffices to work with the form $\lambda \omega_v + \pi^*\omega_S$, which in turn defines the same symplectic connection as $\omega_v$.

Let $D\subset S-M$ be a small disk. Local parallel transport maps for $X_{\phi}|_D \rightarrow D$ are clearly well-defined since the map is proper, but it is not obvious that these preserve the divisor at infinity and hence restrict to give maps on $Y_{\phi}$. One can appeal to a relative version of Moser's theorem to deform the parallel transport vector fields so that they preserve the divisor at infinity, or one can estimate the horizontal vector fields on the open part directly.   With respect to the vertical K\"ahler metric on $Y_{\phi}$ induced by $\|\cdot\|$, the horizontal lift of the vector $\partial_z \in TD$ is
\begin{equation} \label{Eqn:HorizLift}
\partial_z^{\sharp} \ = \ \frac{\nabla \pi}{\| \nabla \pi\|^2}
\end{equation}
since 
\[
D\pi (\partial_z^{\sharp}) = \frac{D\pi(\nabla \pi)}{\| \nabla \pi \|^2}  = 1 \qquad \textrm{and} \ \langle \partial_z^{\sharp}, v\rangle = \frac{D\pi(v)}{\| \nabla \pi \|^2} = 0 \quad \textrm{for} \ v\in T^{vt}(X_{\phi})|_D = \ker(d\pi).
\]
Over $D$, the divisor $\Delta_{\infty}$ is smooth, irreducible and of multiplicity $1$. Choose local co-ordinates $\bf{x}$ near a point $0\in\Delta_{\infty}$ with
\[
\Delta_{\infty} = t^{-1}(0) = \{x_1 = 0\} ; \ \pi(\textbf{x}) = x_n
\]
and write $\|\cdot\| = e^{\sigma}|\cdot|$ as a multiple of the standard metric, for some locally bounded positive function $\sigma$. Then $\omega_{res} = dd^c h$ for a K\"ahler potential
\[
h = -log\|t\|^2 = -log|t|^2-\sigma
\]  
and
\[
|\partial_z^{\sharp}\cdot h| = \frac{|\langle \nabla \pi, \nabla \sigma \rangle|}{\|\nabla \pi\|^2} + \frac{2|t|\, |\langle \nabla t, \nabla \pi \rangle|}{\|\nabla \pi\|^2 \, |t|^2} \ \leq \ \frac{const.}{|\pi|}
\]
which ensures integrability of the horizontal vector field on $Y_{\phi}$ itself.
\end{proof}

The parallel transport maps are not compactly supported, but the image under parallel transport along $\gamma$ of any closed exact Lagrangian submanifold of $\pi^{-1}(\gamma(0))$ is well-defined up to compactly supported Hamiltonian isotopy inside $\pi^{-1}(\gamma(1))$.  One can slightly generalise the story to allow parallel transport along \emph{vanishing} paths $\gamma_0$ which end at a zero of $\phi$, i.e. critical point of the Lefschetz fibration $\pi|_{\pi^{-1}(S-M)}$. In particular, there are well-defined Lefschetz thimbles associated to such paths, in the usual way; see \cite{Seidel:LES} for details.

\subsection{Universal families}\label{Sec:Families}

Recall from the Introduction the finite $2^d$-fold cover 
\begin{equation} \label{Eqn:SignedQuad}
\Quad^{\pm}(\S,\M)_0 \longrightarrow \Quad(\S,\M)_0
\end{equation}
of signed quadratic differentials. Geometrically on $S$, one interprets the sign at $p$ as a choice of residue $\pm\int_{\beta_p} \sqrt{\phi}$, where $\beta_p$ is a small loop encircling $p$ on $S$. The unbranched cover \eqref{Eqn:SignedQuad} extends to $\Quad(\S,\M)$ as a branched cover, with isotropy group $\Z_2^s$ over differentials with $s$ simple poles. 

Both $\Quad(\S,\M)_0$ and its finite cover are complex analytic orbifolds, excluding a handful of exceptional cases where there is a non-trivial generic isotropy group; when $g(\S)>0$ this only occurs if $g(\S)=1$ and $|\M|=1$, when differentials have generic $\bZ_2$-stabiliser. One can also consider \emph{framed} quadratic differentials $\Quad^{\Gamma}(\S,\M)_0$ in which one fixes a framing of the group $H_3(Y_{\phi})$.  A framed quadratic differential determines a signed differential, see \cite{BrSm}, and $\Quad^{\Gamma}(\S,\M)_0$ is smooth.   We will view the choice of sign at a double pole which enters into a signed differential as a choice of component $\bC^2_{+,x} \subset \pi^{-1}(x) \subset Y_{\phi}$ of the reducible fibre over $x$ (making this association canonical will not be required in the sequel).  There is then a universal smooth family
\begin{equation} \label{Eqn:UnivFamily}
\widehat{\scrX} \rightarrow \Quad^{\Gamma}(\S,\M)_0
\end{equation}
of small resolutions over framed quadratic differentials.

\begin{Lemma}\label{Lem:WellDefined}
For any fixed $\lambda > 0$ and $\delta > 0$ sufficiently small, the symplectic six-manifold $(Y,\omega)$ underlying $(Y_{\phi}, \omega_{res}^{\lambda, \delta})$ depends up to symplectic diffeomorphism only on the underlying smooth data $(\S,\M)$.\end{Lemma}

\begin{proof}
Fix $\lambda > 0$ and $\delta>0$ sufficiently small. Note $\omega_{res}^{\lambda,\delta}$ is K\"ahler on $\hat{X}_{\phi}$ for every $\phi$. The period map on $\Quad^{\Gamma}(\S,\M)_0$ equips that space with a flat K\"ahler structure, cf. \cite[Theorem 1]{BrSm}, and  there is a K\"ahler form on the total space of $\widehat{\scrX} \rightarrow \Quad^{\Gamma}(\S,\M)_0$ which restricts to $\omega_{res}$ on each fibre.  Since both varieties $\hat{\Delta}_{\infty} \subset \hat{X}_{\zeta}$ are smooth for every $\zeta$, we can now apply a relative version of Moser's theorem (or an argument as in Lemma \ref{Lem:Parallel})  to the universal family \eqref{Eqn:UnivFamily} to symplectically identify the complements $\hat{X}_{\zeta} \backslash \hat{\Delta}_{\infty} = Y_\zeta$ for different $\zeta$. \end{proof}

To simplify notation, we will write $\omega$ or $\omega_{\lambda}$ to denote a symplectic (K\"ahler) form on $Y_{\phi}$ induced as above by the ample divisor $[\pi^{-1}(pt)] + \lambda \Delta_{\infty}$ on $X_{\phi}$.

\subsection{Lagrangian spheres} \label{Sec:LagSphere}
The general fibre of $(Y_{\phi}, \omega)$ is a finite type Stein manifold, symplectomorphic to a disk cotangent subbundle $DT^*S^2 \subset T^*S^2$, equipped with the restriction of the canonical symplectic form.   A well-known theorem of Hind \cite{Hind} asserts that there is a unique Lagrangian 2-sphere in $T^*S^2$  up to Hamiltonian isotopy. That uniqueness leads to various constructions of 3-dimensional Lagrangian submanifolds in $Y_{\phi}$.

  Pick a path $\gamma: [0,1] \rightarrow S$ with $\{\gamma(0), \gamma(1)\} \subset \Zer_{\phi}$, $\gamma(0)\neq \gamma(1)$, and with $\gamma|_{(0,1)} \subset S\backslash \{M \cup \Zer_{\phi}\}$.  We require the tangent vector of $\gamma$ to be non-trivial at each end-point. 

\begin{Lemma} \label{Lem:Matching}
Such a $\gamma$ defines a Lagrangian sphere $L_{\gamma} \subset Y_{\phi}$, well-defined up to Hamiltonian isotopy.   \end{Lemma}

\begin{proof}
Suppose $\phi_0$ is a meromorphic quadratic differential with a zero of multiplicity two at a point $p \in S$.  The corresponding  3-fold 
\[
Y_{\phi_0} \ = \ \{ det_W =  \pi^* \phi_0\}  
\]
is given locally by a family of quadrics
\[
\{ab-c^2 =  t^2\}_{|t|<\varepsilon} \subset \bC^3\times\bC_t
\]
which has a 3-fold ordinary double point at the origin.  
The sphere $L_{\gamma} \subset Y_{\phi}$ arises as a vanishing cycle for the associated nodal degeneration corresponding to a path of quadratic differentials from $\phi$ to $\phi_0$ for which two simple zeroes of $\phi$ coalesce to a double zero along $\gamma$.  The sphere is well-defined up to Hamiltonian isotopy by a Moser-type argument, starting from the fact that the stratum of quadratic differentials with one double zero (and all other zeroes simple) is connected.
\end{proof}

The path $\gamma$ defines a matching cycle $L_{\gamma}' \subset Y_{\phi}$, fibred over the arc $\gamma$ via Donaldson's construction, cf.  \cite[Section 16]{FCPLT} and \cite{AMP}. In general, the matching cycle construction requires a perturbation of  the symplectic connexion over $\gamma$ and thus of $\omega$, which may in general change its cohomology class (pulling back a non-trivial multiple of the area class on $S$). To avoid this issue, we impose additional symmetry. 

\begin{Lemma} \label{Lem:MatchingPath}
Given any open subset $M \subset U \subset S$,  there is a Hamiltonian isotopy $h_t$, $0\leq t\leq 1$ and $h_0=\id$,  of $Y_{\phi}$  for which $h_1(L_{\gamma})$ fibres over $\gamma$ for every matching path $\gamma \subset S-U$. \end{Lemma}

\begin{proof}
There is a distinguished trivialisation of $\mathcal{O}(M)$ over $S\backslash M$. Pick a spin structure on $S$ viewed as a square root  $K_S^{\frac{1}{2}}$ of the canonical bundle. We may then suppose that  the original vector bundle $V \rightarrow S\backslash M$ is a direct sum of line bundles
\[
V = K_S^{\frac{1}{2}} \oplus K_S^{\frac{1}{2}} = K_S^{\frac{1}{2}}\otimes \bC^2.
\]
Then there is a canonical action of $SU(2)$ by bundle automorphisms of $V|_{S-M}$, hence of $S^2V$, and the determinant map  is $SU(2)$-equivariant. Since $S^2V$ and $W$ are isomorphic over $S\backslash M$, the same holds for the determinant map on $W$.  It follows that, for $\phi \in H^0(K_S(M)^{\otimes 2})$, the hypersurface
\begin{equation} \label{Eqn:WrongHypersurface}
det^{-1}(\phi) \subset W
\end{equation}
has a holomorphic $SU(2)$-action fibrewise over $S\backslash M$ (which does not extend to the total space of $Y_{\phi}$). Away from $\pi^{-1}(U)$, the K\"ahler form on $W$ and hence $Y_{\phi}$ can be  made $SU(2)$-invariant (in fact the $SU(2)$-action factors through $SO(3)$).  

The $SO(3)$ action on any fibre $T^*S^2 \supset D^{< \mu}(S^2) = \pi^{-1}(x)$ is the canonical action induced by rotations of $S^2$.  Since the action is fibrewise, symplectic parallel transport maps are $SO(3)$-equivariant, which in turn means that the vanishing cycles for arbitrary matching paths contained in $S\backslash U$ are $SO(3)$-invariant Lagrangian spheres in $T^*S^2$. However, there is a unique such sphere (any one is an orbit of the action, so distinct ones would be disjoint).  Therefore, for the invariant symplectic form, the matching cycle $L_{\gamma}'$ can be constructed without perturbing the symplectic connection. The Hamiltonian isotopy $h_t$ of the Lemma arises from interpolating a given symplectic form with one which is $SO(3)$-invariant over $S\backslash U$.
  \end{proof}
  
  The previous construction of a Hamiltonian representative for $L_{\gamma}$ fibred over $\gamma$ depends on choices; fortunately, we will not need to carry this out in families.  A choice of orientation for the vanishing cycle in the fibre and of the matching path in $S$ defines an orientation of the Lagrangian $L_{\gamma}$.

\subsection{Lagrangian cylinders}\label{Sec:LagCyl}
Consider a loop $\sigma \subset S$ encircling a point $x\in M$.  Each fibre $\pi^{-1}(y)$ for $y\in \sigma$ contains a Lagrangian 2-sphere, unique up to Hamiltonian isotopy. In particular, if one parallel transports a given 2-sphere $V \subset \pi^{-1}(\sigma(0))$ around $\sigma$, the resulting monodromy image $h_{\sigma}(V)$ is Hamiltonian isotopic to $V$, and co-incident with $V$ if the symplectic form is $SO(3)$-invariant over $\sigma$ as in the previous section. This constructs a Lagrangian submanifold $L_{\sigma} \subset Y_{\phi}$ fibred over $\sigma$.

\begin{Lemma}
$L_{\sigma} \cong S^1\times S^2.$
\end{Lemma}

\begin{proof}
We claim the monodromy around $x\in M$ is Hamiltonian isotopic to the identity. The smooth projective 3-fold $\hat{X}_{\phi} \rightarrow S$ has singular fibre  a union of two rational surfaces meeting along a smooth curve $C\cong \bP^1$, cf. Lemma \ref{Lem:SmallRes}.  This is a Morse-Bott Lefschetz degeneration of the generic smooth fibre $\bP^1\times \bP^1$, so the monodromy is a fibred Dehn twist in the vanishing cycle \cite{Perutz}. In this case, the isotropic fibres of the vanishing cycle are circles in $\bP^1$, so the fibred Dehn twist is Hamiltonian isotopic to the identity. \end{proof}

It will be helpful to be more explicit.  We keep the previous notation.

\begin{Lemma} \label{Lem:ExplicitCylinder}
Let $\Gamma$ be defined by 
\begin{equation} \label{Eqn:LagCylExplicit}
\Gamma = \left\{ x\in \bR, \, p \in e^{-i\theta/2} \bR, \, q\in e^{-i\theta/2}\bR, \, \delta=e^{i\theta} \right\}_{0\leq \theta\leq 2\pi}
\end{equation}
inside the affine variety 
\[
\scrX \ = \ \{ \, x^2 + \delta (p^2+q^2) = 1 \, \} \subset \bC^4.
\]
There is a symplectic structure $\omega_{\scrX}$ on $\scrX$, compatible with the standard complex structure,  for which $\Gamma$ is Lagrangian, and a symplectic embedding of an open subset of the preimage of $\{|\delta| = 1\} \subset (\scrX, \omega_{\scrX})$ into $(Y_{\phi},\omega)$ taking $\Gamma$ to  $L_\sigma$.
\end{Lemma}

\begin{proof}
After flattening the Hermitian metric on the vector bundle $W$ to look like a product near $M$, and deforming $\phi$ through smooth sections to be constant over a small neighbourhood of $M$, one obtains the following symplectic (but not holomorphic) local model for a neighbourhood of a reducible fibre of $Y_{\phi}$. Consider
\[\bP^3 \times \bC \ \supset \ \left\{\,  [x:y:z:t], \delta \, \big| \, x^2 - \delta yz = t^2 \, \right\} \qquad \textrm{with} \ \omega_{FS} \oplus \omega_{std}
\]
and the hyperplane section $\Pi = \{t=0\}$.  The subspace living over some small ball $0\in B_{\delta} \subset \bC_{\delta}$ embeds into $X_{\phi}$, with the local fibration given by projection $\pi$ to the $\delta$-plane, and a symplectic model for $Y_{\phi}$ near a double pole of $\phi$ is the affine complement of $\Pi$,
\begin{equation} \label{Eqn:LocalModel}
\bC^3 \times \bC \ \supset \ \left\{ \, x^2 - \delta yz = 1 \, \right\}.
\end{equation}
A unitary change of  co-ordinates and conformal rescaling gives the hypersurface
\[
\scrX \ = \ \{ \, x^2 + \delta (p^2+q^2) = 1 \, \} \subset \bC^4 
\]
which we equip with the restriction of the standard flat symplectic structure, scaled in the $\delta$-direction so $B_{\delta}$ contains the unit circle. We deform the standard symplectic form $\omega_{\bC^3} \oplus \omega_{\delta}$ by deforming $\omega_{\delta}$ to a form $\omega_{\delta}'$ which coincides with the form $(1/r)\,dr\wedge d\theta$ in a small open neighbourhood of the unit circle.  Since this co-incides with the usual form $r\,dr\,d\theta$ on $S^1$, a sufficiently small such perturbation will still tame the standard integrable complex structure.  There is an antiholomorphic involution
\[
(x,p,q,\delta) \mapsto (\bar{x}, \overline{\delta p}, \overline{\delta q}, 1/\overline{\delta})
\]
which preserves $\scrX$ and reverses the sign of the deformed symplectic form on the submanifold $\scrX|_{\{|\delta|=1\}}$. The fixed locus therefore defines a Lagrangian submanifold
\[
\Gamma = \left\{ x\in \bR, \, p \in e^{-i\theta/2} \bR, \, q\in e^{-i\theta/2}\bR, \, \delta=e^{i\theta} \right\}_{0\leq \theta\leq 2\pi}
\]
as given in the Lemma. This is an $S^2$-bundle over the unit $\delta$-circle; since it is globally Lagrangian, the 2-sphere fibres are  preserved by parallel transport by \cite[Lemma 16.3]{FCPLT}, which implies that $\Gamma = L_{\sigma}$ for $\sigma$ a parametrization of the unit $\delta$-circle. 
\end{proof}

\subsection{WKB collections of spheres} \label{Subsec:WKB}
Fix a complete saddle-free quadratic differential $\phi$ which defines a non-degenerate WKB triangulation in the sense of Lemma \ref{Lem:GoodWKB}. Lemma \ref{Lem:Matching} associates to every edge of the dual cellulation a Lagrangian sphere, which gives a collection of $n=6g-6+3d$ Lagrangian spheres $\{L_e\} \subset Y_{\phi}$.  We will call such a collection a \emph{WKB collection} of Lagrangian spheres.  

We remark that if the differential $\psi$ yielded a WKB triangulation $T$ which contained a self-folded triangle, the dual cellulation would contain an edge with goes from a zero to itself, and the matching sphere construction would yield an \emph{immersed} Lagrangian sphere in $Y_{\psi}$.  The quiver prescription of Labardini-Fragoso then involves replacing this immersed sphere with an embedded replacement, as in Figure \ref{Fig:ImmersedLag}.  In any case, the situation covered by Lemma \ref{Lem:GoodWKB} will suffice in this paper. 
\begin{center}
\begin{figure}[ht]
\includegraphics[scale=0.4]{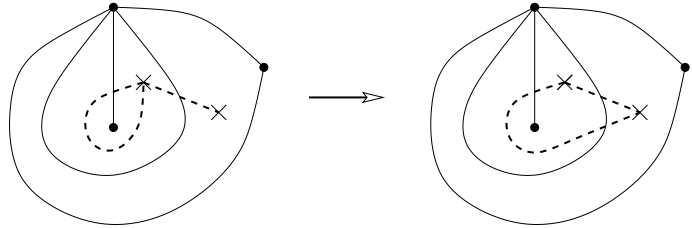}
\caption{A self-folded triangle, its enclosing triangle, and matching spheres (dotted). On the left, one sphere would be immersed; the right shows its embedded WKB replacement.\label{Fig:ImmersedLag}}
\end{figure}
\end{center}

Saddle-free differentials form chambers which are separated by walls on which, in the simplest instance, there is a unique saddle connection; the corresponding triangulations differ by a flip, and the Lagrangian cellulations and WKB collections differ as in Figure \ref{Fig:LagCellFlip}, assuming the WKB triangulations are non-degenerate on both sides of the wall. 
\begin{center}
\begin{figure}[ht]
\includegraphics[scale=0.25]{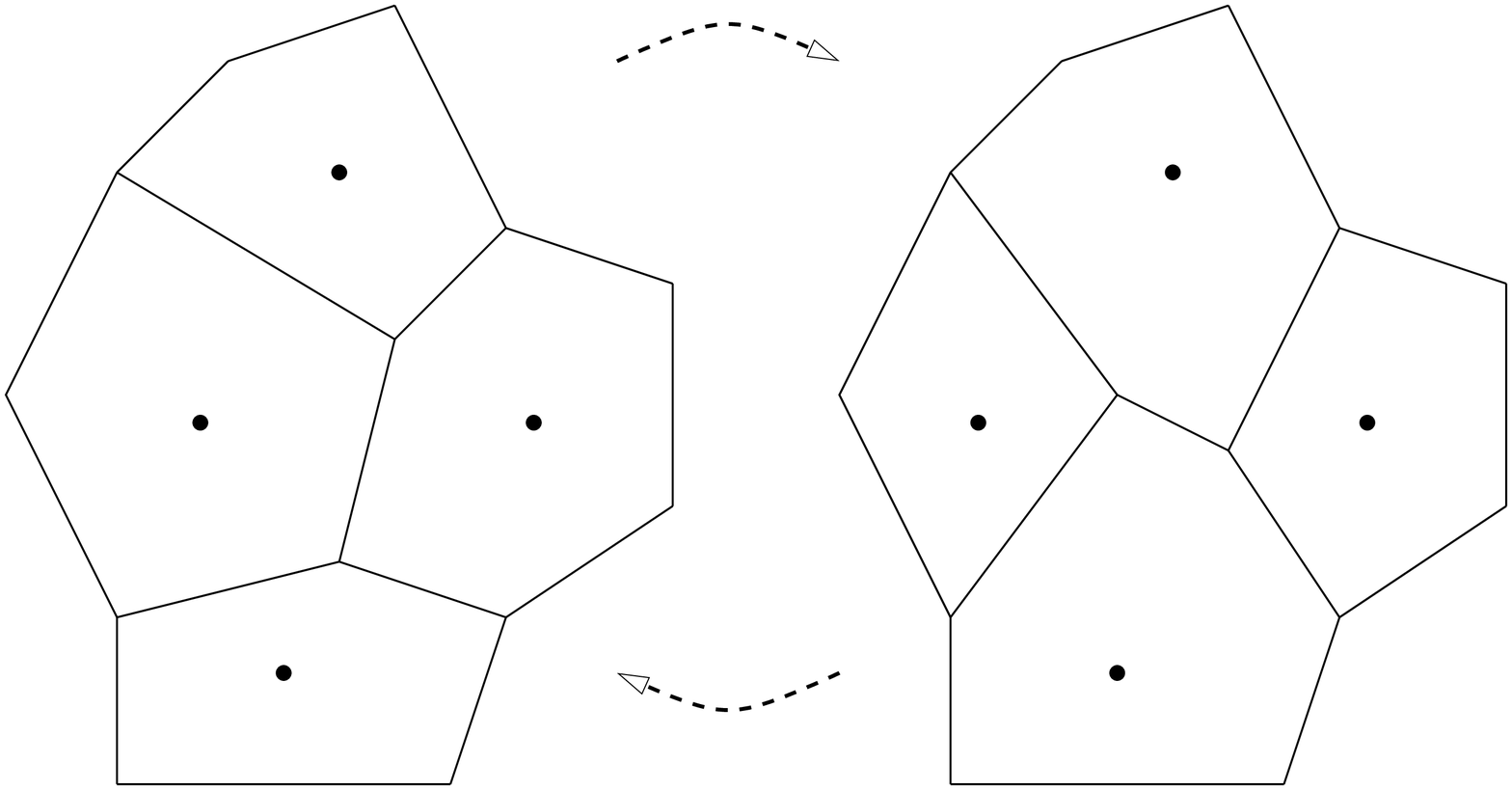}
\caption{Lagrangian cellulations differing by a flip\label{Fig:LagCellFlip}}
\end{figure}
\end{center}
Note that the Lagrangians $L_e$, $L_f$ for two distinct edges may meet at two points, cf. Figure \ref{Fig:WKBLagsg=1}. In this case, the union of the two matching paths $\gamma_e \cup \gamma_f$ is necessarily a homotopically non-trivial loop, by the uniqueness of geodesic representatives for homotopy classes in complete flat surfaces, see \cite{Strebel}.
\begin{center}
\begin{figure}[ht]
\includegraphics[scale=0.35]{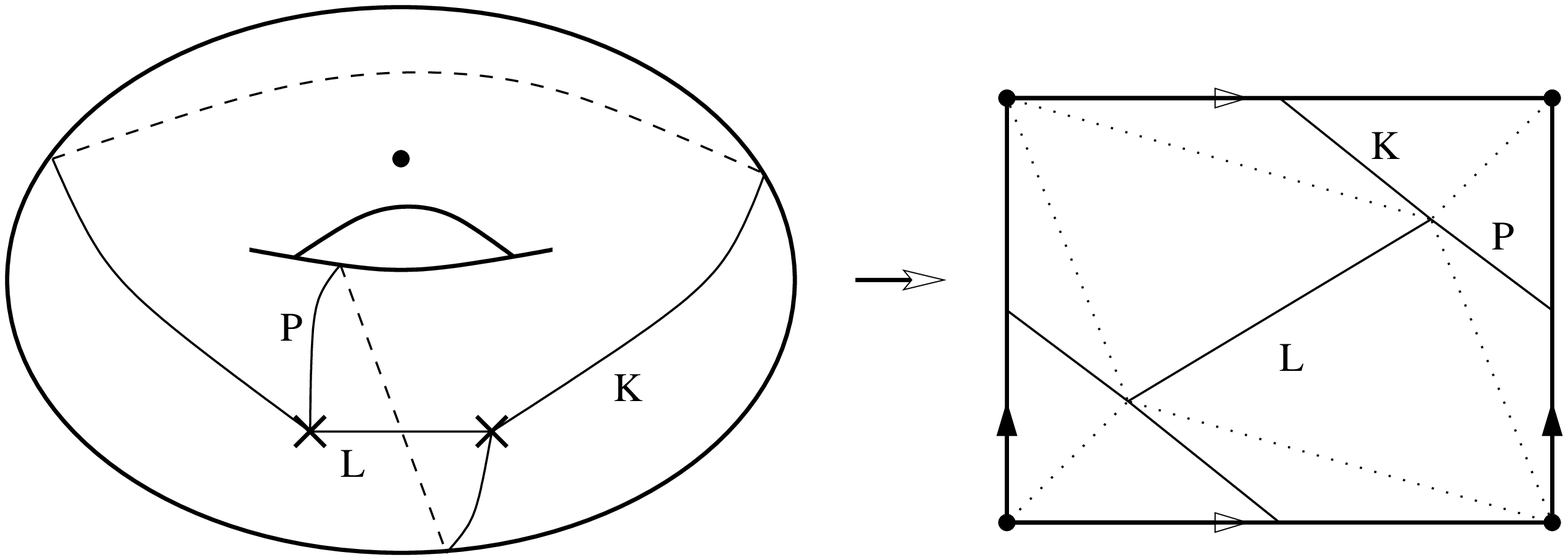}
\caption{WKB Lagrangians $\{L,K,P\}$ in the case $g=1$, $d=1$\label{Fig:WKBLagsg=1}}
\end{figure}
\end{center}

\begin{Remark} \label{Rem:NoTrivalentChance}
A holomorphic quadratic differential $\psi$  on $S$ has $4g-4$ zeroes. There is no trivalent cellulation of $S$ with vertices the zeroes of $\psi$, for reasons of Euler characteristic (trivalence implies the number of faces would be zero).  Thus, there is no direct analogue of the Lagrangian cellulation for the 3-fold $Y'_{\psi}$ of Equation \eqref{Eqn:Holo3fold}.
\end{Remark}

\subsection{Gradings}\label{Sec:Gradings}

Let $(Z,\omega)$ be a symplectic manifold with $2c_1(Z)=0$, so that $Z$ has trivial bicanonical bundle $K_Z^{\otimes 2} \cong \calO$, where $K_Z$ is defined with respect to any compatible almost complex structure. The space of possible homotopy classes of trivialisation of $K_Z^{\otimes 2}$ is given by $H^1(Z;\bZ)$.  Pick a quadratic volume form $\Theta \in H^0(K_Z^{\otimes 2})$ giving the trivialisation.  $\Theta$ defines a map from the Lagrangian Grassmannian to the circle
\[
\alpha: Gr_{Lag}(Z) \rightarrow S^1, \quad \Lambda \mapsto \frac{\eta(v_1 \wedge \ldots \wedge v_n)^2}{|\eta(v_1\wedge\ldots \wedge v_n)^2|}, \ \Lambda = \langle v_j\rangle \subset TZ_z.
\]
For any $L\subset Z$ there is an induced map $\alpha_L: L \rightarrow S^1$ and 
a \emph{grading} of $L\subset Z$ is given by a phase function $\tilde{\alpha}_L: L \rightarrow \bR$ with $exp(2i\pi \tilde{\alpha}_L) = \alpha_L$.  If $\tilde{L}$, $\tilde{L}'$ are graded Lagrangians, any isolated transverse intersection point  $p$ of $L$ and $L'$ acquires an absolute Maslov index $i(\tilde{L}, \tilde{L}'; p) \in \bZ$.  

\begin{Example} \label{Ex:GradeCotangent}
Suppose $L'$ is the graph of an exact one-form $df$ in $T^*L$, and $f$ is Morse with an isolated critical point at $0$.  Equip $L$ with the constant trivial phase function.  There is a distinguished choice of grading on $L'$ compatible with the canonical isotopy from $L'$ to $L$ via graphs of $\varepsilon df$, and with respect to this grading, $i(\tilde{L}, \tilde{L}'; 0)$ is given by the Morse index of $0$ as a critical point of $f$. 
\end{Example}

The Lagrangian submanifolds $L_\gamma$ admit gradings with respect to the holomorphic volume form $\kappa_{\phi}$ of Lemma \ref{Lem:Gradable}, since they are simply-connected.

Since $\phi$ equips the surface $S$ with a flat metric with singularities, a curve $\gamma: [0,1] \rightarrow S-M$ has a well-defined phase at each point with $\gamma(t) \not \in \Zer_{\phi}$.  Recall that $\gamma$ is a geodesic for the $\phi$-metric precisely  if this phase is constant on each connected component of $\gamma^{-1}(S\backslash \Zer_{\phi})$, and any primitive saddle connection for $\phi$ is a geodesic.  

\begin{Lemma} \label{Lem:Gradable2}
There is a volume form $\tilde{\kappa}_{\phi}$ homotopic to $\kappa_{\phi}$ with the property that the phase function of the matching sphere $L_{\gamma} \subset Y_{\phi}$, computed with respect to $\tilde{\kappa}_{\phi}$,  is equal to the $\phi$-phase of the curve $\gamma \subset S$. In particular, saddle connections define Lagrangian 3-spheres of constant phase.
\end{Lemma}

\begin{proof} Away from a neighbourhood $M \subset U \subset S$ we have a holomorphic $SU(2)$-action on $Y_{\phi}$ fixing the divisor at infinity, hence the associated holomorphic volume form is $SU(2)$-invariant in this subset. It follows that the phase function on a matching sphere $L_e$ is $SO(3)$-invariant, where $SO(3)$ rotates the $S^2$-fibres, hence defines a function on the underlying matching path $\gamma_e$.  We claim that  this function co-incides with the phase of $\gamma_e$ in the $\phi$-metric.

The result is local, so after passing to a cover of $S$ we can reduce to the case where $\gamma$ is an arc in the base of a Lefschetz fibration with fibres given by affine quadrics.    An explicit formula for the phase can then be obtained from the Poincar\'e residue theorem, compare to \cite[Section 6]{ThomasYau} or \cite[Section 5e]{KhovanovSeidel}.  In particular, for the local model 
\[
\bC^3 \rightarrow \bC, \qquad (z_1,z_2,z_3) \mapsto \sum z_i^2 = t
\]
equation (6.3) of \cite{ThomasYau} asserts that the phase function associated to $\Theta = dz_1 \wedge dz_2 \wedge dz_3$ at any tangent vector to the Lefschetz thimble defined by a path $\gamma(t)$, and projecting to  $\gamma'(t) \partial_t$, has phase $\gamma'(t) \sqrt{t}$, i.e. that $\Theta$ pushes forward to the one-form $\sqrt{t}\, dt$. Thus $\Theta^{\otimes 2}$ defines a quadratic differential $t \, dt^{\otimes 2}$ on the $t$-plane with a simple zero. 
\end{proof}

The previous Lemma can be used to fix the phase of $\kappa_{\phi}$ uniquely.

\section{Floer theory}

\subsection{Almost complex structures}

The manifold $(Y_{\phi}, \omega)$ is not convex or of contact type at infinity, because the divisor $\Delta_{\infty}$ is not ample.  Lemma \ref{Lem:NoCurve} nonetheless gives good control on holomorphic curve theory in $Y_{\phi}$.

\begin{Definition}
Let $\scrJ_{\phi}$ denote the space of almost complex structures on $Y_{\phi}$ which 
\begin{enumerate}
\item tame the symplectic form $\omega$;
\item  make projection $Y_{\phi}\rightarrow S$ holomorphic;
\item co-incide outside a compact set with the restriction of the integrable complex structure from the crepant resolution $\hat{X}_{\phi}$.
\end{enumerate}
\end{Definition}

\begin{Lemma} \label{Lem:NoRationalCurve}
For $J\in \scrJ_{\phi}$ there is  no non-constant $J$-holomorphic map $\bP^1 \rightarrow Y_{\phi}$, and if $L_\gamma \subset Y_{\phi}$ is a matching sphere, then $L_{\gamma}$ bounds no $J$-holomorphic disk.
\end{Lemma}

\begin{proof}
Since projection $Y_{\phi} \rightarrow S$ is holomorphic and $g(S)>0$, the first statement follows.  The same argument implies that any $J$-holomorphic disk with boundary on $L_\gamma$ is contained in a fibre of the projection $Y_{\phi} \rightarrow S$, but the intersection of $L_{\gamma}$ with any fibre it meets is exact.
\end{proof}

A \emph{pseudoholomorphic disk} denotes the solution to a perturbed Cauchy-Riemann equation 
\begin{equation} \label{Eqn:CR}
(du - \gamma \otimes X_H)^{0,1} \, = \, 0
\end{equation}
defined on a disk with boundary punctures and Lagrangian boundary conditions $L_j \subset int(A)$, where $\gamma \in \Omega^1(D)$ is a 1-form and $X_H$ the Hamiltonian vector field  of a Hamiltonian function $H:\hat{X}_{\phi} \rightarrow \bR$  which vanishes to order at least 2 on the divisor $\hat{\Delta}_{\infty}$.  The $(0,1)$-part of the 1-form is taken with respect to a family of almost complex structures induced by a mapping of the domain of $u$ into $\scrJ_\phi$. Note that the map $u: D \rightarrow \scrJ_\phi$ has the property that $u(t)$ is some fixed integrable structure $J_0$ near $\Delta_{\infty}$.  In local co-ordinates Equation \ref{Eqn:CR} has the form
\[
\partial_s u + J_{u(z)} \partial_t u -  \gamma(\partial_t) \cdot J_{u(z)} \cdot X_H(u, t) \, = \, 0.
\]
Since $X_H \equiv 0$ near $\Delta_{\infty}$ and $J_{u(z)}$ is  constant  near infinity,  outside a relatively compact subset $U$ whose interior contains all the $\{L_j\}$, Equation \ref{Eqn:CR} reduces to the usual unperturbed holomorphic curve equation.   

\begin{Lemma} \label{Lem:Compactness}
Let $\{L_1,\ldots, L_k \} \subset Y_{\phi}$ be  compact Lagrangian submanifolds, and suppose $u_j$ is a sequence of pseudoholomorphic disks  with uniformly bounded energy and with boundary on $\cup_j L_j$. Then the $u_j$ are contained in some compact subset of $Y_{\phi}$.
\end{Lemma}

\begin{proof}
Suppose the conclusion of the Lemma fails. By Gromov compactness in the smooth variety $\hat{X}_{\phi}$, some subsequence of the $u_j$ converges to a curve $u_{\infty}$ in $\hat{X}_{\phi}$ which has non-trivial intersection with the divisor $\hat{\Delta}_{\infty}$.  Since the Cauchy-Riemann equation is unperturbed near infinity, the image  of $u_{\infty}$ must meet $\hat{\Delta}_{\infty}$ locally positively, except for components contained inside the divisor.  More precisely, positivity of intersections applied to the principal component of the stable curve limit means that this limit curve must contain at least one bubble component which is a rational curve in $\hat{\Delta}_{\infty}$ with strictly negative intersection with $\hat{\Delta}_{\infty}$. No such curves exist by Lemma \ref{Lem:NoCurve}.
\end{proof}

\subsection{Fukaya category generalities}\label{Sec:Fukaya}

The strictly unobstructed Fukaya categories occuring in this paper belong to a technically manageable regime.  The relevant transversality theory is encompassed by material in \cite{FCPLT, Seidel:Flux}, to which we defer for essentially all details of the construction.   

 Let $(Y, \omega)$ be a symplectic manifold which admits a class of taming almost complex structures $\scrJ_Y$ which satisfy the first conclusion of  Lemma \ref{Lem:NoRationalCurve}.  We only consider strictly unobstructed Lagrangian submanifolds, and suppose furthermore that the conclusion of Lemma \ref{Lem:Compactness} is valid.   (A quasi-projective variety $Y$ admitting a compactification as in Lemma \ref{Lem:NoCurve} is the most relevant source of examples.)  For ach $b\in H^2(Y;\bZ_2)$ there is a $\bZ$-graded $A_{\infty}$-category $\scrF(Y;b)$, linear over $\Lambda_{\bC}$, called the strictly unobstructed ($b$-twisted) Fukaya category.  Objects of $\scrF(Y;b)$ are Lagrangian branes $L\subset Y$, namely:

\begin{itemize}
\item $L\subset Y$ is a closed oriented Lagrangian submanifold;
\item $J_L \in \scrJ_Y$ is an almost complex structure for which $L$ bounds no $J_L$-holomorphic disk and meets no $J_L$-holomorphic sphere;
\item $L$ carries a relative spin structure, relative to the class $b$;  
\item $L$ is graded, see Section \ref{Sec:Gradings}.
\end{itemize}

Morphisms in $\scrF(Y;b)$ are given by the Floer cochain complex $(CF^*(L, L'), \mu^1)$, which is freely generated by intersection points of $L$ and $L'$ if they intersect transversely.   More properly, index theory and the choice of gradings on $L$, $L'$ associate to any isolated transverse intersection point $x \in L \pitchfork L'$ an abstract one-dimensional $\Lambda_{\bC}$-vector space $\mathrm{or}_x$, see \cite[Section 11h]{FCPLT}, and the Floer complex
\[
CF^*(L,L') \ = \ \oplus_x \,  \mathrm{or}_x.
\]
There are higher order chain-level operations which comprise a collection of maps 
\begin{equation} \label{Eqn:Ainfty}
\mu_{\scrF}^d: CF(L_{d-1},L_d) \otimes \cdots \otimes CF(L_0,L_1) \rightarrow CF(L_0,L_d)[2-d]
\end{equation}
 of degree $2-d$, for $d\geq 1$, with $\mu_{\scrF}^1$ being the aforementioned differential and $\mu_{\scrF}^2$ the holomorphic triangle product.
The $\{\mu^d_{\scrF}\}$ have matrix coefficients which are defined by counting holomorphic disks with $(d+1)$-boundary punctures, whose arcs map to the Lagrangian submanifolds $(L_0,\ldots, L_d)$ in cyclic order and which converge in strip-like end co-ordinates at the punctures to intersection points.   The moduli spaces of disks are naturally oriented relative to the orientation lines occuring in \eqref{Eqn:Ainfty} (in a manner which depends on the choice of $b$), so the count of pseudo-holomorphic disks is a signed count.  The count of a disk $u$ is  weighted by the symplectic area $q^{\int_u \omega}$, with $q$ the Novikov parameter.

The construction of the operations $\mu^d_{\scrF}$ is rather involved, and we defer to \cite[Section 3]{Seidel:Flux} for details;  in particular, the coefficients $\mu^d(x_{d-1},\ldots,x_0)$ depend on additional perturbation data $(\EuScript{K}, \EuScript{J})$ (choices of Hamiltonian functions, domain-dependent almost complex structures, strip-like ends etc; these choices in part overcome the difficulty that a Lagrangian is never transverse to itself). The coefficients are not individually well-defined (the $\mu^k$ are not chain maps), but the entire structure is invariant up to a suitable notion of quasi-isomorphism.   Hamiltonian isotopic Lagrangian submanifolds, equipped with brane data compatible with the isotopy, define quasi-isomorphic objects of $\scrF(Y;b)$.

We denote by $\Tw\,\scrF(Y;b)$ the category of twisted complexes over $\scrF(Y;b)$, and by $\Tw^{\pi}\,\scrF(Y;b)$ its idempotent completion.  The corresponding cohomological categories  are denoted $\D\scrF(Y;b)$ and  $\D^{\pi}\scrF(Y;b)$.

We record one particular fact for later.

\begin{Lemma} \label{Lem:TwistEffect}
Let $b \in H^2(Y;\bZ_2)$ be supported by a locally finite cycle $F_b \subset Y$ disjoint from a collection of spin Lagrangian submanifolds $\{L_i\}_{0\leq i \leq k}$. Suppose the $L_i$ are pairwise transverse, and fix intersection points $x_i \in CF(L_i, L_{i+1})$ with cyclic indices.   If a rigid holomorphic disk $u$ contributes to the coefficient $\mu^k_{\scrF(Y)}(x_{k-1},\ldots, x_0)$ with value $\kappa$, then it contributes to the same coefficient in $\scrF(Y;b)$ with coefficient $(-1)^d\cdot\kappa$, where $d = u \cdot F_b$ is the algebraic intersection number of the disk with the cycle $F_b$. 
\end{Lemma}

\begin{proof}
See \cite[Vol. II, Proposition 8.1.16]{FO3}. Since $F_b \cap L_i = \emptyset$, each of the spin Lagrangians $L_i$ is also relatively spin relative to $b$. Note that two Hamiltonian isotopic representatives for $L$ in $Y$ each lying in $Y \backslash F_b$, which define quasi-isomorphic objects of $\scrF(Y;b)$, may not be Hamiltonian isotopic in $Y\backslash F_b$. The trace on $L$ of the  isotopy with the cycle $F_b$  defines a $1$-cycle in $L$, Poincar\'e dual to a class in $H^2(L;\bZ_2)$ which twists its relative spin structure compatibly with the change in intersection number with $F_b$ of some given element in $\pi_2(Y, L)$. \end{proof}

If one encounters Lagrangians in some convenient geometric position (clean Morse-Bott intersections, matching cycles in a Lefschetz fibration) it is often useful to compute without perturbing them.   Given a finite set of Lagrangians $\{L_j\}$ which meet pairwise transversely, one can define the corresponding Fukaya category subject to this constraint, but the stucture coefficients are obtained from (virtual) counts of more general objects called  \emph{pearly trees}, see  \cite[Section 7]{Seidel:HMSgenus2} and \cite[Section 4]{Sheridan}.  

First, one defines $CF^*(L_i, L_i) = C^*_{Morse}(f_i)$ for a fixed Morse function $f_i: L_i \rightarrow \bR$ (Morse-Smale for an underlying Riemannian metric).   An abstract  pearly tree is a planar tree $\Gamma \subset \bR^2$ with one infinite incoming and several infinite outgoing edges, and  $d\geq 0$ finite-length internal edges, vertices of valence at least 3, the connected components of the complement  $\bR^2 \backslash \Gamma$ being labelled by Lagrangians $L_j$.  A holomorphic pearly tree comprises a collection of pseudoholomorphic disks and gradient flow-lines, satisfying obvious incidence and compatibility conditions; gradient flowlines arise when computing a higher product  $\mu^k(x_{k-1},\ldots,x_0)$ for which some inputs $x_i \in CF^*(L_i,L_i)$, see Figure \ref{Fig:PearlyTree}. 
\begin{center}
\begin{figure}[ht]
\includegraphics[scale=0.4]{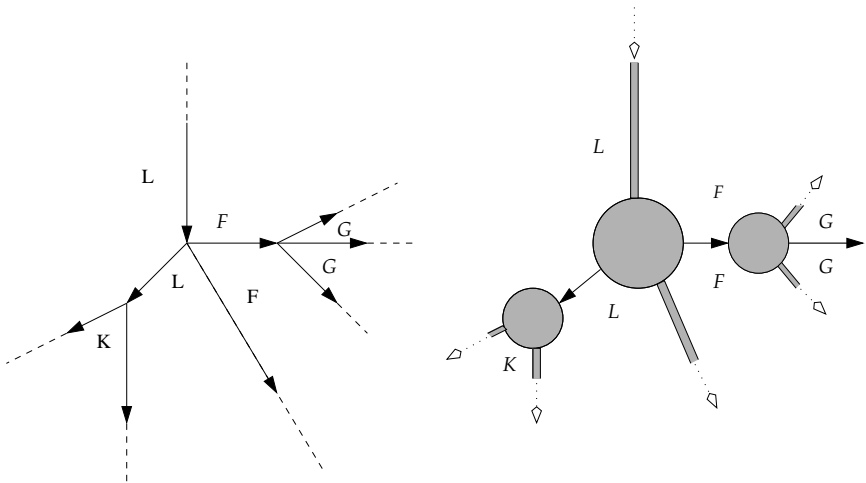}
\caption{An abstract planar tree, and associated pearly configuration\label{Fig:PearlyTree}}
\end{figure}
\end{center}

There are two important situations in which one can avoid pearls  for purposes of computing a coefficient of \eqref{Eqn:Ainfty}. 
\begin{enumerate}
\item If \emph{adjacent} boundary conditions $\{L_i, L_{i+1}\}$ (with cyclic indices) are always pairwise transverse, in particular never co-incide, there are no pearly contributions to this particular coefficient of $\mu^k$.  This relies on an important theorem due to Sheridan \cite[Proposition 4.6]{Sheridan}: moduli spaces of pearls can be made regular by generic choices of consistent perturbation data, and in the regular case pearls with $d$ internal edges form a stratum of real codimension $d$. Therefore  for isolated regular pearls, there are no internal (finite length Morse) edges of the underlying planar tree.

\item If there is exactly one adjacent pair of co-incident Lagrangians $L_i = L_{i+1}$, and the corresponding input or output $x_i \in HF^*(L_i,L_i) \cong H^*(L_i)$ is the class of top degree, then one can count pseudoholomorphic polygons which are smooth at the given corner but have an incidence condition, passing through a fixed generic point $q \in L_i$ Poincar\'e dual to $x_i$.  Compare to \cite[Section 7]{Seidel:HMSgenus2}.
\end{enumerate}

The choice of $b_0$ as background class is relevant in Lemma \ref{Lem:LocalNonzero2}, but much of the discussion in the next sections applies to the categories $\scrF(Y;b)$ uniformly.  We will sometimes omit the background class $b \in H^2(Y_{\phi};\bZ_2)$ from our notation when it plays no role.

\subsection{Grading the WKB algebra} \label{Subsec:Grading}

Now return to the 3-fold $Y_{\phi}$. Any Lagrangian matching sphere is strictly unobstructed and admits a unique spin structure.   Since $L$ bounds no holomorphic disks, the cohomology $H(CF^*(L,L)) \cong H^*(L)$, equipped with its classical $A_{\infty}$-structure.

Given a finite collection of Lagrangian spheres  $\{L_e\} \subset Y_{\phi}$, and a choice of $b \in H^2(Y_{\phi};\bZ_2)$, there is an associated total morphism $A_{\infty}$-algebra 
\begin{equation} \label{Eqn:WKBMorphismAlgebra}
\scrA = \scrA_b = \oplus_{e,e'} HF(L_e, L_{e'})
\end{equation}
Theorem \ref{Thm:Main2} involves computing the $A_{\infty}$-algebra $\scrA = \scrA(T;b_0)$, where the indices $e \in T$ are indexed by edges of a non-degenerate triangulation $T$ and correspond to a WKB-collection of Lagrangian spheres, and identifying it with the corresponding Ginzburg potential algebra (i.e. with the total endomorphism algebra of the category $\CC$ considered in Section \ref{Sec:QuiverPotential}).

Lemma \ref{Lem:MatchingPath} implies that each $L_e$ can be taken to fibre over the path $\gamma_e$, and hence it suffices to compute the Floer theory amongst such a collection of fibred matching spheres. By projecting to $S$, pseudoholomorphic disks are then constrained by the Riemann mapping theorem. Collections of matching spheres don't lie in general position (there are triple intersections at vertices of the Lagrangian cellulation), so in principle one must define $\scrA$ via pearls, but the remarks at the end of Section \ref{Sec:Fukaya} imply that in the case at hand the consequences are fairly benign.

 Consider the Lagrangian submanifolds $\{L_e\}$ which are matching spheres for the edges of the Lagrangian cellulation (dual to a non-degenerate WKB triangulation).  Any two distinct Lagrangians $L_e, L_f$ are either disjoint or meet at either one or two isolated points, lying over trivalent vertices of the cellulation (the nodal points of a fibre of $\pi$ lying over a zero of $\phi$).

\begin{Lemma} \label{Lem:ConcentrateGrading}
The $L_e$ admit gradings for which 
\begin{itemize}
\item the algebra $\scrA = \oplus_{e,e'} HF^*(L_e, L_{e'})$ is concentrated in degrees $0\leq \ast \leq 3$; 
\item the isolated intersection points have absolute Maslov index $+1$ clockwise and $+2$ anticlockwise, cf. Figure \ref{Fig:MaslovGrading}.
\end{itemize}
\end{Lemma}

\begin{center}
\begin{figure}[ht]
\includegraphics[scale=0.5]{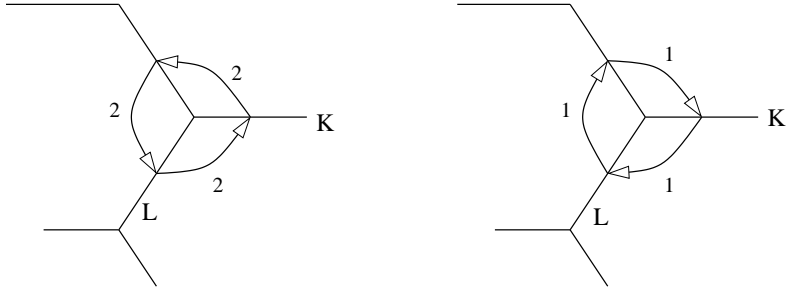}
\caption{A point contributes to $CF^2(L,K)$ and $CF^1(K,L)$.\label{Fig:MaslovGrading}} 
\end{figure}
\end{center}

\begin{proof}
The groups $HF^*(L,L) \cong H^*(S^3)$ carry their natural grading, and by Poincar\'e duality $HF^*(L,L') \cong HF^{3-*}(L',L)$ for any $L,L'$, so it suffices to determine the grading of an isolated intersection point lying over a  zero of $\phi$.  The Lagrangian matching paths of a WKB-type  Lagrangian cellulation are realised by geodesics for the $\phi$-metric on $S$, so the Lagrangians are locally given by transversely intersecting special Lagrangian thimbles.

More explicitly, working locally near a simple zero of the quadratic differential $\phi$, we consider the three straight arcs of the associated vertical foliation, which form the terminals of a trivalent vertex.  (Comparing to Figure \ref{Fig:LagCellFlip}, the leaves of the horizontal foliation at a zero fall into the double poles at the centres of  the three cells adjacent to the zero, and the edges of the Lagrangian cellulation are given, locally near the zero, by leaves of the vertical foliation.)  Each of the arcs defines by parallel transport a Lagrangian disk in the 3-fold $Y_{\phi}$, and Lemma \ref{Lem:Gradable2} implies these all have identical  phase.  There are Darboux co-ordinates in which these three Lagrangians are given by linear subspaces $\bR^3, e^{i\pi/3}\bR^3$ and $e^{2i\pi/3}\bR^3 \subset \bC^3$. (Note the quadratic volume form $\Theta = (dz_1 \wedge dz_2 \wedge dz_3)^{\otimes 2}$ is invariant under rotation by $\pi/3$, just as the quadratic volume form on $S$, locally given by $t\, dt^{\otimes 2}$ with $t= \sum z_j^2$, is locally invariant under rotation by $2\pi/3$.)

$L'=e^{i\pi/3}\bR^3$ is the graph of the differential of a function over $L=\bR^3$ with an isolated minimum, so a Morse critical point of index $0$.  Therefore, for the grading on $L'$ compatible with the obvious rotation isotopy back to $L$, the absolute index would be zero by Example \ref{Ex:GradeCotangent}.  The phase function $\alpha_{L'} \equiv \pi/2$ differs from the phase function compatible with that isotopy by the constant function $1$, hence the index of the intersection point is $+1$.   (An alternative for the last step is to use non-vanishing of the triangle product, Lemma \ref{Lem:LocalZero} below, to show that since the absolute indices are symmetric under rotation of Figure \ref{Fig:MaslovGrading} by $2\pi/3$, they must all equal $+1$.)
  \end{proof}

The local Morse-theoretic description of an isolated intersection $x$ of WKB spheres $L, L'$ given above also yields preferred isomorphisms $\mathrm{or}_x \cong k$ between the orientation lines and the ground field, coming from preferred trivialisations $\det(D_x) \cong k$ for the determinant line of a $\cdbar$-operator on a half-plane with linear Lagrangian boundary conditions which rotate by $\pi/3$. Via these trivialisations, Lemma \ref{Lem:ConcentrateGrading} shows that the WKB algebra is isomorphic, as a graded vector space, to the total morphism algebra of the category $\CC$ introduced in Section \ref{Sec:QuiverPotential}.  

\subsection{First constraints on polygons}

Let $\{L_e\}$ be matching spheres which are edges of a non-degenerate WKB cellulation. Appealing to \cite[Lemma 2.1]{FCPLT}, we can take the $A_{\infty}$-structure on the algebra $\scrA$ from \eqref{Eqn:WKBMorphismAlgebra} to be strictly unital. 

\begin{Lemma} \label{Lem:GradingInputs}
Let $k\geq 3$.  If the product
\[
\mu^k: HF^*(L_{k-1},L_k) \otimes \cdots \otimes HF^*(L_0,L_1) \rightarrow HF^*(L_0,L_k)[2-k]
\]
is non-zero, then either $L_0 \neq \L_k$ and all inputs have degree $1$, or $L_0 = L_k$, exactly one input has degree $2$ and all others have degree $1$.
\end{Lemma}

\begin{proof}
Lemma \ref{Lem:ConcentrateGrading} implies that the degree $0$ subalgebra of the WKB algebra is spanned by the units $e_i \in HF^0(L_i, L_i)$ of the constituent WKB spheres.  It follows that 
since the $A_{\infty}$-structure is strictly unital, none of the inputs to a non-trivial operation $\mu^k$ with $k\geq 3$ has degree $0$. Hence every input has degree $\geq 1$, whilst $\mu^k$ has degree $2-k$. Since $\scrA$ is concentrated in degrees $0 \leq \ast \leq 3$, it follows that no input can have degree $3$, and at most one input can have degree $2$. Moreover, there is an input of degree $2$ if and only if the output has degree $3$, which is possible only if $L_0$ and $L_k$ co-incide. 
\end{proof}

\begin{Lemma} \label{Lem:All-inputs-degree-one}
The second case of Lemma \ref{Lem:GradingInputs} does not occur.
\end{Lemma}

\begin{proof}
Working with pearls, we take a Morse model $CF^*(L,L) = C^*_{Morse}(L)$  for self Floer cochains; without loss of generality, for each WKB sphere we can take a perfect Morse function so the Floer cochain group has rank 2.  If there is a degree two input to the product in Lemma \ref{Lem:GradingInputs}, the output is in the rank one space $CF^3(L_0,L_0)$, which means that the holomorphic disk should pass through the stable manifold for the gradient flow of the maximum of the Morse function on $L_0$. Therefore this marked point is unconstrained, and hence no non-constant disk can be rigid. It follows from the proof of Lemma \ref{Lem:LocalZero} below that constant polygons contribute non-trivially to $\mu^2$ but not to $\mu^k$ for $k\geq 3$. \end{proof}

According to \cite[Theorem 1.1]{Fukaya:cyclic}, over $\Lambda_{\bC}$ the Fukaya category can always be taken both cyclic and strictly unital, and Lemma \ref{Lem:All-inputs-degree-one} would also follow formally from that fact, compare to Section \ref{Sec:QuiverPotential}.  A helpful consequence of the previous result is that one can compute the Fukaya category using pseudoholomorphic disks rather than pearly trees.  Note that these results do not imply that all the $\{L_i\}$ in \eqref{Eqn:Ainfty} are pairwise distinct for the corresponding operation to be non-trivial, see Figure \ref{Fig:HoloT2} for an example.
\begin{center}
\begin{figure}[ht]
\includegraphics[scale=0.35]{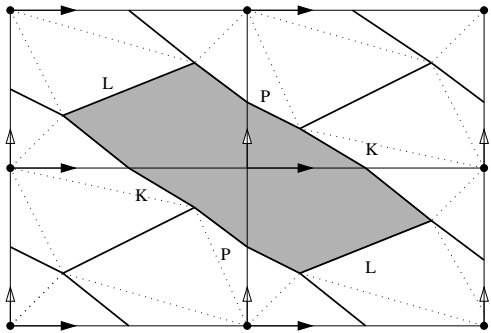}
\caption{A polygon contributing to $\mu^5$ in the universal cover when $g(S)=1, |M|=1$; note the holomorphic disk maps onto $S=T^2$. \label{Fig:HoloT2}}
\end{figure}
\end{center}

The boundary of any holomorphic polygon contributing to $\mu^k$ defines a (not necessarily embedded) parametrized closed path in the graph on the surface $S$ formed by the edges of the Lagrangian cellulation.  Inputs of degree $1$ correspond to turning clockwise in the graph defined by the cellulation edges, and inputs of degree $2$ correspond to turning anticlockwise, along the boundary of the disk; it is simplest to see this by lifting the disk to the universal cover, given by pulling back the fibration $\pi: Y_{\phi} \rightarrow S$ to the universal cover of $S$, where the Lagrangian cellulation edges give rise to a trivalent planar graph. Lemma \ref{Lem:GradingInputs} implies that there is at most one anticlockwise turn.  

\begin{Example}\label{Ex:PD} Suppose $L$ and $K$ meet transversely at a single point. The unique closed path of length $2$ gives rise to the product 
\begin{equation} \label{Eqn:Product}
HF^1(L,K) \otimes HF^2(K,L) \rightarrow HF^3(K,K)
\end{equation}
which is  non-trivial by Poincar\'e duality. 
\end{Example}

\subsection{Constant triangles}

Let $L_0,L_1,L_2 \subset Z$ be three graded Lagrangian submanifolds, intersecting pairwise transversely at a point $p$. There is a constant holomorphic triangle $u$ with boundary conditions (anticlockwise ordered) $L_0,L_1,L_2$. Let $D_u$ be the linearized operator at $u$, i.e. the $\bar\partial$-operator on the trivial vector bundle with fibre $TZ_p$ over a disk with three boundary punctures, with boundary values in $TL_{0,p}$, $TL_{1,p}$, $TL_{2,p}$. The index formula for such operators \cite[Proposition 11.13]{FCPLT} implies that 
\begin{equation} \label{eq:triangle}
\index(D_u) = i(\tilde{L}_0,\tilde{L}_2;p) - i(\tilde{L}_0,\tilde{L}_1;p) - i(\tilde{L}_1,\tilde{L}_2;p).
\end{equation}

We now return to the situation (and notation) arising in the proof of Lemma \ref{Lem:ConcentrateGrading}.  We label the linear Lagrangians $\bR^3, e^{i\pi/3}\bR^3$ and $e^{2i\pi/3}\bR^3 \subset \bC^3$ by $L_0$, $L_{2\pi/3}$ and $L_{4\pi/3}$, corresponding to the slopes in the complex plane of the arcs over which they fibre in the model Lefschetz fibration
\[
p: \bC^3 \rightarrow \bC, \quad (z_1, z_2, z_3) \mapsto z_1^2 + z_2^2 + z_3^2.
\]
For the zero phase functions of Lemma \ref{Lem:ConcentrateGrading}, the  indices appearing in \eqref{eq:triangle} are 
\[
 i(\tilde{L}_0,\tilde{L}_{2\pi/3};p) =2; \  i(\tilde{L}_0,\tilde{L}_{4\pi/3};p) =1; \  i(\tilde{L}_{4\pi/3},\tilde{L}_{2\pi/3};p) = 1,
\]
so the constant triangle has index $0$ and can in principle contribute non-trivially to the product in the algebra $\scrA_b$.  

\begin{Lemma} \label{Lem:LocalZero}
The constant holomorphic triangle at a trivalent zero contributes exactly $+1$ to a non-trivial multiplication
\[
\mu^2_{\scrF}: HF^1(L_{4\pi/3}, L_{2\pi/3}) \otimes  HF^1(L_0, L_{4\pi/3})  \ \longrightarrow \ \ HF^2(L_0, L_{2\pi/3}).
\]
\end{Lemma}

\begin{proof}
The local model splits into a direct sum of 3 copies of the geometry given by 3 real lines in $\bC$ passing through the origin. 
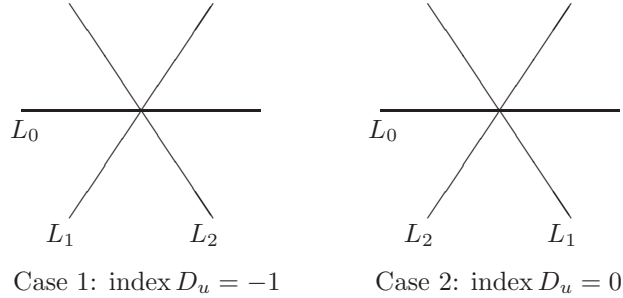
\begin{figure}[ht]
\setlength{\unitlength}{3947sp}%
\begingroup\makeatletter\ifx\SetFigFont\undefined%
\gdef\SetFigFont#1#2#3#4#5{%
  \reset@font\fontsize{#1}{#2pt}%
  \fontfamily{#3}\fontseries{#4}\fontshape{#5}%
  \selectfont}%
\fi\endgroup%
\begin{picture}(3837,1861)(226,-1310)
\put(2520,-1261){\makebox(0,0)[lb]{\smash{{\SetFigFont{10}{12.0}{\rmdefault}{\mddefault}{\updefault}{Case 2: $\index D_u = 0$}%
}}}}
\thinlines
{\put(601,-811){\line( 2, 3){900}}
}%
{\put(301,-136){\line( 1, 0){1500}}
}%
{\put(2851,539){\line( 2,-3){900}}
}%
{\put(2851,-811){\line( 2, 3){900}}
}%
{\put(2551,-136){\line( 1, 0){1500}}
}%
\put(226,-306){\makebox(0,0)[lb]{\smash{{\SetFigFont{10}{12.0}{\rmdefault}{\mddefault}{\updefault}{$L_0$}%
}}}}
\put(451,-961){\makebox(0,0)[lb]{\smash{{\SetFigFont{10}{12.0}{\rmdefault}{\mddefault}{\updefault}{$L_1$}%
}}}}
\put(1351,-961){\makebox(0,0)[lb]{\smash{{\SetFigFont{10}{12.0}{\rmdefault}{\mddefault}{\updefault}{$L_2$}%
}}}}
\put(2476,-306){\makebox(0,0)[lb]{\smash{{\SetFigFont{10}{12.0}{\rmdefault}{\mddefault}{\updefault}{$L_0$}%
}}}}
\put(2701,-961){\makebox(0,0)[lb]{\smash{{\SetFigFont{10}{12.0}{\rmdefault}{\mddefault}{\updefault}{$L_2$}%
}}}}
\put(3601,-961){\makebox(0,0)[lb]{\smash{{\SetFigFont{10}{12.0}{\rmdefault}{\mddefault}{\updefault}{$L_1$}%
}}}}
\put(246,-1261){\makebox(0,0)[lb]{\smash{{\SetFigFont{10}{12.0}{\rmdefault}{\mddefault}{\updefault}{Case 1: $\index D_u = -1$}%
}}}}
{\put(601,539){\line( 2,-3){900}}
}%
\end{picture}%
\caption{Indices of constant holomorphic triangles\label{Fig:2-triangles}}
\end{figure}
The regularity of the constant triangle $u$ is then standard, provided the Lagrangians are taken in the appropriate order. For a constant triangle $u$ in $\bC$ contributing to a product 
\[
HF^*(L_1,L_2) \otimes HF^*(L_0,L_1) \longrightarrow HF^*(L_0,L_2)
\]
the indices are as given in Figure \ref{Fig:2-triangles} (consider perturbing the three lines to create a non-trivial triangle in $\bC$, which is either holomorphic or antiholomorphic depending on the cyclic order of the boundary conditions).
\end{proof}

Constant triangles and Poincar\'e duality do not always completely determine the algebra structure in $\scrA$; there may be additional products if some of the cells of the Lagrangian cellulation are themselves triangles.

\subsection{Holomorphic disks on totally real cylinders}

Constant holomorphic triangles provide the cubic terms in the superpotential associated to a triangulated surface arising from the ``inscribed triangles" of Labardini-Fragoso's quiver prescription for a potential \cite{LF1}, see \eqref{Eqn:PotentialForNondegenQuiver} and Figure \ref{fig1}.  
The higher order terms of the potential arise from higher order $A_{\infty}$-products in $\scrA$. En route to computing these, it will be helpful to study a simpler local model involving the Lagrangian cylinder of Section \ref{Sec:LagCyl}.

Recall the totally real model $S^1\times S^2 \cong \Gamma \subset \bC^4$ for the Lagrangian cylinder given in \eqref{Eqn:LagCylExplicit}.  We will consider holomorphic sections $D \rightarrow \bC^4$ with this boundary condition.

\begin{Lemma} \label{Lem:LocalNonzero}
In the notation of Lemma \ref{Lem:ExplicitCylinder}, there are exactly two rigid holomorphic sections of $\pi: \scrX \rightarrow D$ with boundary condition on $\Gamma$, and all other moduli spaces of sections are empty.
\end{Lemma}

\begin{proof}
Consider a section $D \rightarrow \scrX$ of $\pi$.   Since $\scrX \subset \bC^4$, each co-ordinate projection defines a holomorphic map $D \rightarrow \bC$. In the $x$ co-ordinate this is a map from the unit disk $D_{\delta}$ to $\bC$ which takes the boundary circle to the real line; any such is constant, so in fact the non-trivial geometry takes place in the $\bC^2$-bundle over the disk with co-ordinates $p,q$.  In each of these factors, the total space of the fibrewise Lagrangian of Equation \ref{Eqn:LagCylExplicit} over the boundary circle defines a M\"obius band, for instance
\[
\bC^2_{p,\delta} \ \supset \ \left\{ \, \delta = e^{i\theta}, p \in e^{-i\theta/2}\bR \,\right\}
\]
 Consider a holomorphic section of $\pi$ in this factor.  This defines a map $z \mapsto v(z)$ on the disk with the property that $v(z) \in z^{-1/2}\bR$ for $z\in S^1$.  \cite[Lemma 2.16]{Seidel:LES} shows that the only holomorphic maps $u: D\rightarrow \bC$ satisfying $u(z) \in z^{1/2}\bR$ on $S^1$ are 
\[
u: z\  \mapsto \  cz + \bar{c} \quad  \textrm{for \ some} \ c\in \bC,
\]
Therefore if $v$ is a holomorphic section of $\pi$, then $zv(z) = cz+\bar{c}$, but then $v(z) = (c+\frac{\bar{c}}{z})$ is meromorphic with a pole, unless $c=0$.  It follows that the only holomorphic sections of $\pi$ with boundary condition $\Gamma$ are the constant sections
\[ x(\delta) = \pm 1; \ p(\delta) = 0 =q(\delta).
\]
Since the sections are constant, they are both regular.   \end{proof}

We  equip $\Gamma$ with the unique spin structure which is bounding in the circle factor.  The moduli space of holomorphic sections of $\pi: \scrX \rightarrow D$ with boundary on $\Gamma$ inherits a canonical orientation.  (More precisely, we are interested in the space of sections with a single boundary marked point, which has a natural evaluation map to $\Gamma$.)  Note that this spin structure  is in particular a relative spin structure relative to the class $b_0$ of \eqref{Eqn:BackgroundCycle}, since ${b_0}|_{\Gamma} = 0$. 
The choice of background class $b\in H^2(Y_{\phi};\bZ_2)$ enters the computation of $\scrA(T;b)$ at  the following point.

\begin{Lemma} \label{Lem:LocalNonzero2}
 The two rigid sections of Lemma \ref{Lem:LocalNonzero} contribute to the moduli space of sections with opposite signs, with respect to the trivial background class $b=0$, and with the same sign, relative to the background class $b=b_0$.
 \end{Lemma}

\begin{proof}
We return to the local model of Lemma \ref{Lem:ExplicitCylinder}, and the affine variety $\{x^2 + \delta(p^2+q^2)=1\} \subset \bC^4$.  
 There is a holomorphic involution $\iota$ of the model given by $x\mapsto -x$, $p\mapsto -p$. This acts fibrewise, preserves the Lagrangian $\Gamma$, and exchanges the two holomorphic sections.  The involution  preserves the orientation of $\Gamma$, which implies that the sections count with the same respectively opposite signs (for the trivial background class) depending on whether or not the involution does or does not preserve the stable trivialisation of the Lagrangian boundary condition along the sections determined by the spin structure on $\Gamma$, cf. \cite[Proof of Theorem 8.1.1]{FO3}.  We claim that the involution does \emph{not} preserve this stable framing.
  
Fix a model $S^1 \times S^2 \subset \bR^2 \times \bR^3 = \bR^5$, with the bounding spin structure.  Let $(x,y,z)$ denote co-ordinates on the $\bR^3$-factor. There is a natural parametrisation $S^1 \times S^2 \rightarrow \Gamma$ taking $(e^{i\theta}, (x,y,z))$ to 
\[
(e^{i\theta}, (x, e^{-i\theta/2}(\cos(\theta/2)y + \sin(\theta/2)z), e^{-i\theta/2}(\sin(\theta/2)y - \cos(\theta/2)z) = (e^{i\theta}, (x,p,q)).
\]
 The involution  $\iota$ acts fibrewise on the $S^2$-factor of $S^1 \times S^2$ by reversing the sign of $x$ and reflecting the circle in the $(y,z)$-plane along an axis of angle $\theta/2$.  Altogether, this means we act by the generator of $\pi_1(SO(3))$ and hence of $\pi_1(SO(5))$.  Put differently, the boundary values $x = \pm 1, p=0=q$ of the holomorphic sections, viewed as $S^1$-submanifolds of $\Gamma$, have an obvious ``constant" framing (constant in a twisted parametrisation)  from the tangent spaces to the $(p,q)$-plane.  These ``constant" framings are exchanged by the involution, but are not simultaneously compatible with a choice of spin structure on $\Gamma$.  Indeed, $\Gamma$ is really the mapping torus of a rotation by $\pi$ on $S^2$, and the two holomorphic sections have boundary values coming from the two rotation fixed points.  Homotoping the monodromy to the identity rotates the tangent spaces at the fixed points  by $\pi$ in opposite directions, hence can bring only one of the two ``constant" framings to the framing induced by the fixed spin structure. The upshot is that the two sections contribute with opposite signs to the moduli space oriented with respect to the trivial background class.
 
 Note that $b_0$ is supported by a locally finite cycle $H$ disjoint from the Lagrangian $\Gamma$. The two holomorphic sections of Lemma \ref{Lem:LocalNonzero} each hit one of the two components of the reducible fibre over the local point of $M$, hence exactly one of them intersects $H$, so turning on the background class $b_0$ changes the sign of exactly one of the disks by Lemma \ref{Lem:TwistEffect}.  \end{proof}
 
 
 \begin{Remark} \label{Rem:obstructed}
 After twisting by $b_0$, the boundary values of the sections of Lemma \ref{Lem:LocalNonzero} sweep the 1-cycle 
 \[
 \frak{m}_0(S^1 \times S^2) = \pm 2[\partial \sigma] \neq 0 \in H_1(S^1\times S^2),\]
 which is not the restriction of any central element of $QH^*(M;b_0)$. Therefore a Lagrangian cylinder $L_{\sigma}$  has obstructed $b_0$-twisted Floer cohomology, and does not define an object of $\scrF(Y_{\phi};b_0)$, even if one enlarges the category to allow unobstructed branes in the sense of \cite{FO3} rather than just weakly exact Lagrangian submanifolds.  
  \end{Remark}

\subsection{Nonconstant polygons}

Fix a finite collection $\{L_0,\ldots, L_k\}$ of matching spheres $L_i \subset Y_{\phi}$ arranged cyclically, bounding an open subset of $S$ containing a unique double pole of $\phi$.  The Lagrangians meet at a sequence of intersection points $x_j = L_j \pitchfork L_{j+1}$ for $0\leq j\leq k$ (with $L_{k+1} \equiv L_0$ by definition). Appealing to Lemma \ref{Lem:ConcentrateGrading}, we may suppose that  $x_j$ has Maslov index $1$. Let $\bar{x}_k$ denote the dual Maslov index 2 intersection point which is a generator of $HF^2(L_0, L_k)$.

\begin{Lemma} \label{Lem:HigherProduct}
If the $\{L_j\}_{1\leq j\leq k}$ are pairwise distinct, the coefficient  of $\kappa$ in the product 
\[
\mu^k_{\scrF(Y;b)}(x_{k-1},\ldots, x_0) \ = \ \kappa \,\bar{x}_k
\] is well-defined; it  is independent of the auxiliary choices made in constructing the Fukaya category. \end{Lemma}

\begin{proof}
Recall that the construction of the Fukaya category relies on additional perturbation data $(\EuScript{K}, \EuScript{J})$ of Hamiltonian deformations, domain-dependent almost complex structures or Morse-Smale pairs made universally over moduli spaces of stable disks or pearly trees. 
As in \cite[Section 10e]{FCPLT}, two different choices of that auxiliary data $(\EuScript{K}_i,\EuScript{J}_i)$, $i=1,2$, for defining the $A_{\infty}$-operations can be related by a continuation $A_{\infty}$-functor $\scrG = (\scrG^i)_{i\geq 1}$ between the resulting Fukaya categories $\scrF_i$. Given a configuration of WKB Lagrangians,  $\scrG^1$ (the action on morphism groups) is necessarily the identity, whilst $\mu^1_{\scrF_i}$ also vanishes on $CF(L_i, L_j)$ (defining self-Floer cochains using perfect Morse functions).  The $A_{\infty}$-functor equations then imply that
\begin{equation}
\begin{aligned}
& \mu^k_{\scrF_2}(x_{k-1},\ldots,x_0) + \sum_{i<k,j} \mu^{i}_{\scrF_2}(x_{k-1},\ldots, \scrG^j(x_{i+j},\ldots,x_i), x_{i-1},\ldots, x_1))  \ = \\
\qquad&  \scrG^1(\mu^k_{\scrF_1}(x_{k-1},\ldots,x_0))  + \sum_{j, i<k} \pm \scrG^j(x_{k-1}, \ldots, \mu^i_{\scrF_1}(x_{i+j},\ldots, x_j),\ldots, x_0))
\end{aligned}
\end{equation}
The higher products $\mu^j_{\scrF}$ amongst the $\{x_i\}$ all vanish for $2<j<k$, since any contributing holomorphic disk would not have boundary defining a closed curve in the WKB-graph. That immediately implies the $\mu^k$-product is  invariant under $\scrG$.
\end{proof}

Note that the hypothesis that the $\{L_i\}_{1 \leq i\leq k}$ are pairwise distinct  can be weakened to the hypothesis that no cyclic subchain $\{L_i, \ldots, L_{i+j}\} \subset \{ L_0, \ldots, L_k\}$ arises as the ordered boundary of any non-trivial element of $\pi_2(Y_\phi, \cup_j L_j)$.  Since holomorphic disks lift uniquely to unramified covers, it would also be sufficient to realise either of these hypotheses after pullback to a  covering space, as in Figure \ref{Fig:HoloT2}. 

 Suppose one has a Lefschetz fibration $p: X \rightarrow B$ and a finite collection of $k+1$ cyclically ordered weakly exact matching spheres $\{K_j\} \subset X$ which enclose a disk containing no critical points of $p$. Suppose that the Lagrangians $\{K_j\}$ are pairwise distinct, and have the obvious affine $A_k$-intersection graph.  Let $y_j \in K_j \pitchfork K_{j+1}$ denote the intersection points, for $0\leq j\leq k$ with cyclic indices. The following is well-known, and reflects the existence of a relation between the classes $[K_j] \in K^0(\scrF(X))$.
  
\begin{Lemma} \label{Lem:Krelation}
 If $e_{K_0} \in HF^0(K_0, K_0)$ denotes the unit,
\begin{equation} \label{Eqn:Krelation}
\mu^{k+1}(y_{k}, \ldots, y_0) = \pm e_{K_0}.
\end{equation}
\end{Lemma}

\begin{proof}
The coefficient of $e_{K_0}$ in the output is well-defined by Lemma \ref{Lem:HigherProduct}.  An iterated application of the long exact triangle in Floer cohomology expresses $K_k$ as a twisted complex on $\{K_0, \ldots, K_{k-1}\}$, and the non-triviality of \eqref{Eqn:Krelation} follows from the Maurer-Cartan equation for the differential in that complex.
\end{proof}

\begin{Remark} \label{Rem:GradingMiracle}
Suppose $\psi$ is a \emph{holomorphic} quadratic differential, and consider a cycle of matching spheres in the 3-fold $Y'_{\psi}$ of \eqref{Eqn:Holo3fold} enclosing a disk with no critical points of $\psi$, hence in the configuration of the previous Lemma. The non-vanishing of this higher product implies that $\sum_{i=0}^k |y_i| + 1-k = 0$,  which precludes $|y_i|=1$ for every $i$. It follows that $\scrF(Y'_\psi;b)$ is not described by a quiver with potential of WKB-type. Compare to Remark \ref{Rem:NoTrivalentChance}.
\end{Remark}

The higher product of \eqref{Eqn:Krelation} counts disks with $k+1$ punctures, which extend smoothly across the output puncture  but  are constrained at the corresponding marked point on the boundary of the disk to pass through a generic point  $p \in K_0$, the minimum of the Morse function used to define $CF^*(K_0, K_0)$.  In particular, consider the moduli space $\scrM_{\ast}(y_{k-1},\ldots,y_0)$ of Maslov zero holomorphic polygons with $k$ punctures at the given intersection points, and one boundary marked point on the boundary interval between the first and last inputs.  The $A_{\infty}$-relations then imply that for generic $J$, $\scrM_{\ast}$ is a closed manifold of dimension $n = \dim_{\bR}(K_j)$.  There is an evaluation map
\begin{equation} \label{Eqn:Evaluate}
ev_{\ast}: \scrM_{\ast}(y_{k-1},\ldots,y_0) \longrightarrow K_0
\end{equation}
which is transverse for generic choices of perturbation data, and  Lemma \ref{Lem:Krelation} implies that
\begin{equation}\label{Eqn:NonzeroDegree}
ev_{\ast} \ \textrm{has  degree}  \ \pm 1.
\end{equation}

\begin{Lemma} \label{Lem:LocalPole}
For background class $b_0$, the contribution to the multiplication
\[
\mu^k_{\scrF(Y;b_0)}: HF^1(L_{k-1},L_k) \otimes \cdots \otimes HF^1(L_0,L_1) \ \longrightarrow \ HF^2(L_0,L_k)
\]
from holomorphic disks which project to a single WKB cell with multiplicity one is non-trivial.
\end{Lemma}

\begin{proof}
We are essentially in the situation of Lemma \ref{Lem:HigherProduct}, bearing in mind the remark after that Lemma. Note that our standing hypotheses imply that $|\M| > 1$, so the number of faces of the Lagrangian cellulation associated to a non-degenerate WKB triangulation is at least 2. In particular, given any face of the cellulation, there is at least one boundary edge $\gamma$ on that face  which occurs exactly once in the boundary, separating two distinct embedded open faces on the surface $S$.   (This need not be true when $|\M|=1$, see Figure \ref{Fig:HoloT2}.)

Following a well-known strategy going back  at least to \cite{Seidel:LES}, we then apply a degeneration and  gluing argument to the situation depicted on the left of  Figure  \ref{Fig:GluingArgument}, which is a schematic for a polygon $R$ of matching spheres encircling a double pole in $S$. 
\begin{figure}[ht]
\begin{center}
\includegraphics[scale=0.5]{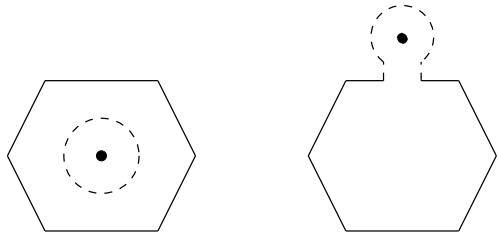}
\end{center}
\caption{Bubbling off sections of the local model near a double pole\label{Fig:GluingArgument}}
\end{figure}
 There is a Lagrangian boundary condition fibred over the dashed circle, and Lemma \ref{Lem:LocalNonzero2} asserts that the algebraic count of isolated holomorphic sections over this inner disk $D_{loc}$  is non-vanishing.   There is a degeneration of the restriction of the fibration to this WKB cell, indicated on the right of Figure \ref{Fig:GluingArgument}, to a pair of fibrations over disks with boundary marked points, namely the local model and the trivial fibration over a polygon $R_{triv}$ encircling no singular fibres.   In this degeneration, we push the double pole towards an edge $\gamma$ which occurs only once on the boundary on the given polygon, so the degeneration on the surface is locally schematically well depicted by Figure \ref{Fig:GluingArgument}.  The gluing theorem \cite[Proposition 2.14]{Seidel:LES} implies that
\begin{equation} \label{Eqn:AfterDegeneration}
\mathcal{M}^0(R) \ \cong \ \amalg_{p+q=2} \mathcal{M}^p(D_{loc}, e) \times_L \mathcal{M}^q (R_{triv}, e)
\end{equation}
where $\mathcal{M}^i$ denotes the $i$-dimensional component of a space of holomorphic sections over  the relevant domain, $e\in\gamma$ is the boundary marked point on the fibrations after degeneration, and the fibre product is taken over evaluation into the Lagrangian 2-sphere $L = L_{\gamma} \cap \pi^{-1}(e) \subset \pi^{-1}(e)$.  The evaluation maps are transverse for generic $J\in\scrJ_\phi$, see e.g. \cite[Proposition 3.4.2]{McD-S}.  On the right of Equation \ref{Eqn:AfterDegeneration}, we have computed already that $\mathcal{M}^p(D_{loc},e)$ is trivial except for $p=0$; on the other hand, Lemma \ref{Lem:Krelation} and \eqref{Eqn:NonzeroDegree} imply that the evaluation map $M^2(R_{triv},e)$ sweeps $L$ with multiplicity 1.  The fibre-product is therefore non-empty; the moduli space on the left of Equation \ref{Eqn:AfterDegeneration} is cobordant to two (transversely cut out) isolated points.  

Since the two holomorphic sections contribute to $\mathcal{M}^0(D_{loc},e)$ with the same sign by Lemma \ref{Lem:LocalNonzero2}, given our choice of background class $b_0$, the holomorphic polygons contribute to $\mu^k$ with the same sign.\end{proof}

If the conic fibres of $Y_{\phi} \rightarrow S$ over points in $M$ had a different topological type (smooth, Lefschetz singularities, higher multiplicity, empty etc),  or if we took $b=0$, it seems the analogue of Lemma \ref{Lem:LocalPole} would not hold.

\subsection{Summary}

Fix a non-degenerate triangulation $T$ of $\S$ which satisfies the conditions of Theorem \ref{Thm:GLFS1}.   Denote by $\CC(T) = \CC(Q(T),W(T))$ the associated Ginzburg category from Section \ref{Sec:QuiverPotential}, for the trivial signing $\epsilon \equiv 1$. Via Lemma \ref{Lem:GoodWKB}, pick a meromorphic quadratic differential $\phi \in H^0(K_S(M)^{\otimes 2})$ whose associated WKB triangulation is $T$.   Let $\{L_e\}$ denote the corresponding WKB Lagrangian  spheres, graded as in Lemmas \ref{Lem:LocalZero} and \ref{Lem:LocalPole}, and denote by $\scrA_{b_0}$ the total morphism algebra $\oplus_{e,f} HF^*(L_e,L_f)$ in the category $\scrF(Y_{\phi};b_0)$. 

\begin{Theorem} \label{Thm:FloerEqualsQuiver}
Suppose $|\M| \geq 3$.  The 
 $A_{\infty}$-algebra $\scrA_{b_0}$ is quasi-isomorphic to $\CC(T)$.  
\end{Theorem}

\begin{proof}
We begin by recalling the discussion from the end of Section \ref{Subsec:Grading}. 
The $A_{\infty}$-algebra $\scrA_{b_0} = \scrA$ is the total morphism algebra of an $A_{\infty}$-category over $\Lambda_{\bC}$ whose objects are the Lagrangians $L_e$, which correspond one-to-one with the vertices of the quiver $Q(T)$. From Lemma \ref{Lem:ConcentrateGrading} we know that each $L_e$ has Floer cohomology isomorphic as a graded algebra to $H^*(S^3)$, and the morphisms between the different $L_e$ are based by the arrows in the quiver $Q(T)$,  graded as in the construction of the category $\CC(T)$ in Section \ref{Sec:QuiverPotential}. Thus, as  graded vector spaces, $\scrA \cong \CC(T)$ are isomorphic.

Floer theory equips $\scrA$ with an $A_{\infty}$-structure which, by homological perturbation, we can take to be minimal and strictly unital.  The structure coefficients of this $A_{\infty}$-structure count holomorphic polygons with boundary conditions the $L_e$, weighted by their symplectic areas.   Lemmas \ref{Lem:GradingInputs} and \ref{Lem:All-inputs-degree-one}, together with Poincar\'e duality as in Example \ref{Ex:PD},  imply that in the notation of Equation \ref{Eqn:CyclicVersion} the only non-trivial multiplications $m_{n-1}$ in $\scrA$ with $n\geq 3$ are exactly those recorded by the corresponding cyclic degree $-n$ maps $c_n$, i.e. the $A_{\infty}$-structure is in fact cyclic.   It follows that the $A_{\infty}$-structure on $\scrA$ may be encoded by the terms of a reduced potential $W$ on $Q(T)$.

Lemma \ref{Lem:LocalZero} implies that there is a non-trivial multiplication $m_2$ in $\scrA$, and hence non-trivial cubic term in the potential $W$, for each isolated intersection point of a triple of Lagrangians $L_e$; indeed, the corresponding Floer product counts a constant disc (of area zero) with co-efficient $+1$.   Such isolated triple intersections exactly correspond to the clockwise 3-cycles $T(f)$ for the faces $f$ of the triangulation $T$, and show that the potential $W = \sum_f T(f) + W''$, where by definition $W''$ is disjoint from $\sum_f T(f)$. More prosaically, disjointness asserts that none of the oriented 3-cycles corresponding to the oriented triangular faces of $Q(T)$ centred on vertices of $T$ occur in $W''$ (each triangular face of $Q(T)$ defines three such 3-cycles, pairwise cyclically equivalent).  

Lemma \ref{Lem:LocalPole} implies that there are further non-trivial multiplications $m_k$ in $\scrA$ for each $(k+1)$-dimensional face of the WKB Lagrangian cellulation on $S$ defined by $\phi$.  The corresponding coefficients $c_{k+1}$ in the potential $W$ are exactly the (anticlockwise-oriented) cycles $C(p)$ occuring in the quiver $Q(T)$ (cf. Figure \ref{fig1}). In the notation of Lemma \ref{Lem:HigherProduct},   Lemma \ref{Lem:LocalPole} implies that the Floer product $\mu^k(x_{k-1},\ldots,x_0)$ is a non-zero multiple of $\bar{x}_k$, where the multiple depends on the symplectic area of the polygons occuring in Lemma \ref{Lem:LocalPole} and the orientation (relative to the orientation lines associated to intersection points) of that moduli space: we have established (via Lemma \ref{Lem:LocalNonzero2}) that the two holomorphic sections which contribute to the given product count with the same sign, hence cannot cancel, but have not pinned down that sign.  At this stage, we may therefore write the potential $W$ as 
\begin{equation} \label{Eqn:GotThere}
W = \sum_f T(f) - \sum_p \lambda_p C(p) + W'
\end{equation}
where we sum over the faces respectively vertices of $Q(T)$, where $\lambda_p = \pm 2 q^{\int_u \omega} \in \Lambda_{\bC}^*$ records the area and orientation of the holomorphic discs $u$ of Lemma \ref{Lem:LocalPole}, and where $W'$ is disjoint from the various cycles $T(f)$ and $C(p)$. Since $[\omega] \in H^2(Y_{\phi}, \cup_e L_e) \cong H^0(S, \cup_e \gamma_e)$ (with the isomorphism arising from the Leray-Serre spectral sequence of the projection map), the choice of symplectic form not only determines but is essentially equivalent to a choice of (non-zero) scalar coefficients $\lambda_p$ above. 

At this point, we have determined the $A_{\infty}$-structure constants which arise either from constant holomorphic polygons or from polygons which project to a single face of the WKB cellulation of $S$ with multiplicity one.  In general there will be additional polygons, for instance those which project to a single cell with higher multiplicity (corresponding to powers $C(p)^j$ of the cycles $C(p)$) or to the union of two adjacent cells (corresponding to a concatenation $C(p)C(q)$), which give rise to the terms in $W'$. Rather than determining these by hand, we appeal to finite determinacy, via the work of Geiss, Labardini-Fragoso and Schr\"oer.   Equation \ref{Eqn:GotThere} shows that the potential $W$ governing the $A_{\infty}$-structure on $\scrA$ is exactly of the form given in Equation \ref{Eqn:General}. By Theorem \ref{Thm:GLFS1}, we infer that the potential $W$ is weakly right-equivalent to $W(T)$, which shows that $\scrA \simeq \CC(T)$ are $A_{\infty}$-quasi-isomorphic as required.
\end{proof}

We should point out that analogous finite determinacy arguments have been used to pin down $A_{\infty}$-structures in symplectic topology elsewhere, going back at least to Seidel's \cite{Seidel:HMSgenus2}.


Since the quiver category $\CC(T)$ is independent, up to derived equivalence, on the particular WKB triangulation $T$, one can infer the same result for the subcategories  $\Tw \scrA_{b_0}$ of $\scrF(Y;b_0)$ generated by the Lagrangians coming from different cellulations (strictly, provided those cellulations come from triangulations satisfying the hypotheses of Theorem \ref{Thm:GLFS1}).  Since $\D(T) = H^0 \Tw(\CC(T))$ is idempotent complete,  Theorem \ref{Thm:FloerEqualsQuiver} then implies that
\[
\Tw^{\pi} \scrA_{b_0} \ \simeq \ \Tw \, \scrA_{b_0}.
\]
Since $\scrA_{b_0}$ is tautologically a subcategory of the Fukaya category $\scrF(Y_{\phi};b_0)$, on 
passing to twisted complexes and then cohomology we see that $\D(\CC(T)) \hookrightarrow \D\scrF(Y_{\phi};b_0)$, which  completes the proof of (the first case of) Theorem \ref{Thm:Main2} as stated in the Introduction.


\subsection{Generation}\label{Sec:Generation}

Fix a non-degenerate triangulation $T$. This subsection outlines one approach to proving that $\scrA(T;b_0)$ generates the subcategory $\scrK(Y_{\phi};b_0) \subset \scrF(Y_{\phi};b_0)$ generated by all matching spheres. This sketch falls far short of a proof, but it seems worth including in part because one sees  the Harder-Narasimhan filtration of a matching sphere with respect to the stability condition determined by $\phi$ emerge as a natural ingredient in the argument.

There are two pieces of  background for the sketch.  Let $L$ and $L'$ denote two matching spheres which meet transversely once, at a vertex $p$ of the cellulation. Consider the Lagrange surgery $L \# L'$ of $L$ and $L'$ at $p$, which is equivalently  given by taking the Dehn twist of $L'$ about $L$.  (There are two local surgeries, corresponding to the positive and negative Dehn twists; here we are taking $p \in HF^1(L, L')$, a morphism from $L$ to $L'$.) Seidel's theorem \cite{Seidel:LES} yields a distinguished triangle in $\scrK(Y_{\phi};b)$
\begin{equation} \label{Eqn:LES}
\xymatrix{
L \ar[rr] & & L' \ar[dl] \\ & L\#L' \ar[ul]^{[1]} &
}
\end{equation}
This gives a long exact sequence of Floer cohomology groups in $Y_{\phi}$
\[
\cdots \rightarrow HF(L'',L) \xrightarrow{\mu^2(p,\cdot)} HF(L'', L') \rightarrow HF(L'', L\#L') \rightarrow \cdots 
\]
for any third Lagrangian submanifold $L'' \in \scrF(Y;b)$, and shows the Lagrange surgery $L \# L'$ is generated by the two constituent Lagrangians $L$ and $L'$.

The second piece of background concerns Floer theory for immersed Lagrangians \cite{Akaho}.  Consider a closed saddle trajectory $\gamma_0$ for $\phi$ of phase $\theta$, going from a zero to itself. If we assume that the residue of $\phi$ at each double pole does not belong to $e^{i\theta}\bR_{<0}$,  then $\gamma_0$ forms one boundary component of a ring domain, the other boundary of which is a union of straight arcs also of phase $\theta$ (the hypothesis on the residues rules out the degenerate case in which the other boundary component collapses onto a double pole).  The closed saddle $\gamma_0$ defines an immersed Lagrangian 3-sphere $L_{\gamma_0} \hookrightarrow Y_{\phi}$. Despite being immersed, this has well-defined Floer cohomology (indeed $L_{\gamma_0}$  is still strictly unobstructed, i.e. bounds no holomorphic disks, so this is straightforward), and
\[
HF(L_{\gamma_0}, L_{\gamma_0}) \cong H^*(S^3) \oplus k^{\oplus 2}
\]
where the second summand arises from the node of $L_{\gamma_0}$.

The suggested route to generation of $\scrK(Y_{\phi};b_0)$ by $\scrA(T;b_0)$ has three ingredients. 

\noindent \textbf{Step 1.}  Let $\gamma_0$ be as above, and suppose for simplicity that both boundaries of the ring domain containing $\gamma_0$ are composed of a single saddle (this is true for generic $\phi$, compare to the discussion of hat-proportional saddles in \cite[Section 5]{BrSm}). 
\begin{figure}[ht]
\begin{center}
\includegraphics[scale=0.4]{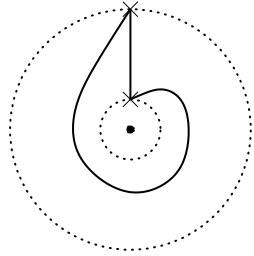}
\end{center}
\caption{A non-degenerate ring domain with a pair of  arcs  $\alpha, \beta$.  \label{Fig:ring}}
\end{figure}
 Let $\alpha, \beta$ be two $\phi$-geodesics inside the ring domain bound by $\gamma_0$, as in Figure \ref{Fig:ring}, defining embedded matching spheres $L_{\alpha}$, $L_{\beta}$ in $Y_{\phi}$ with
\[
HF(L_{\alpha}, L_{\beta}) \cong k\langle x\rangle \oplus k \langle y \rangle
\] 
where $\{x,y\}$ denote the zeroes on the boundary of the ring domain and $x$ belongs to $\gamma_0$.  It seems natural to conjecture that the immersed sphere is quasi-isomorphic to the cone
\begin{equation} \label{Eqn:ImmersedSurgery}
L_{\gamma_0} \simeq (L_{\alpha} \stackrel{y}{\longrightarrow} L_{\beta}).
\end{equation}
This is a version of \eqref{Eqn:LES}, in a setting where the surgery is unobstructed but immersed (in particular cannot be interpreted in terms of a Dehn twist).  One roundabout approach to a proof of \eqref{Eqn:ImmersedSurgery} would appeal to known results on homological mirror symmetry for the 3-dimensional affine $A_2$-singularity, which provides a local model for two spherical objects having a two-dimensional space of extensions, cf. the discussion of the representation theory of the quiver (J1) in \cite[Section 11]{BrSm}. 

\noindent \textbf{Step 2.}  The existence theorem for geodesics in flat surfaces implies that a given matching path $\gamma \subset S$ is homotopic to a concatenation of $\phi$-geodesics $\gamma \simeq \gamma_{j_0} \ast \cdots \ast \gamma_{j_k}$, where each $\gamma_{j_i}$ is embedded in its interior (hence is either a geodesic matching path or a closed geodesic from a zero to itself).  Given the correspondence between geodesics and stable objects \cite[Theorem 1.4]{BrSm}, this decomposition should reflect the Harder-Narasimham filtration of $L_{\gamma}$ in the stability condition associated to $\phi$, and (especially in light of \cite{ThomasYau}) there are obvious potential connections to mean curvature flow.  In particular, one expects that 
$L_{\gamma} \in \langle L_{\gamma_{j_0}}, \ldots, L_{\gamma_{j_k}}\rangle$, where $\langle \cdots \rangle$ denotes the extension-closed subcategory generated by the given objects.  An inductive argument might reduce that claim to the special cases of the surgery exact triangle  discussed above. 

\noindent \textbf{Step 3.} At this point, one would know that any Lagrangian matching sphere is generated by the subset of matching spheres coming from $\phi$-geodesics.   
An embedded $\phi$-geodesic $\tau$ of phase $\theta$ occurs as one of the WKB-Lagrangians for the cellulation associated to $e^{-i(\theta+\delta)}\phi$, for sufficiently small $\delta$,  compare to \cite[Figure 32]{BrSm}.  If the triangulation $T'$ associated to this rotated differential is non-degenerate, it is immediate from Theorem \ref{Thm:Main2} and the fact that $\D(T) \simeq \D(T')$ that $L_{\tau}$ is generated by the category $\scrA(T;b_0)$ associated to the triangulation $T$ defined by $\phi$ itself. 

However, if $e^{-i(\theta+\delta)}\phi$ has a triangulation $T'$ with self-folded triangles, this is more delicate. To conclude that $\scrA(T;b_0)$ generates $\scrK(Y_{\phi};b_0)$, at least by this method, one would need further to introduce the algebras $\scrA(T';b_0)$ for arbitrary, not necessarily non-degenerate triangulations $T'$, and to prove that $\scrA(T';b_0) \simeq \D(T')$ in this wider setting.  Note in particular that any such argument would seem to hinge on an extension of the results of \cite{GLFS} to this more general setting.  Assuming that such an extension had been carried out, however, one could then conclude the desired generation: $L_{\gamma}$ would be generated by geodesics,  geodesics arise as WKB-Lagrangians for some triangulation, and all the subcategories of $\scrF(Y_{\phi};b_0)$ associated to triangulations would be quasi-equivalent.

\section{Miscellany}

\subsection{Higher order poles}

Suppose we have a marked bordered surface $(\S,\M)$ with $\partial \S \neq \emptyset$, arising from a quadratic differential $\phi$ on a Riemann surface $S$ with a non-empty set of poles of order $\geq 3$.  We construct the 3-fold $Y_{\phi} \rightarrow S \backslash \Pol_{\geq 3}(\phi)$ as indicated in the Introduction.  

The preceding arguments need only minimal changes to yield the second case of Theorem \ref{Thm:Main2}. The locally free sheaf $W$ now has first Chern class $3K_S + 2M_{=2} + 3 M_{\geq 3}$, in the notation of \eqref{Eqn:divisorsplits}. The proof of Lemma \ref{Lem:FirstChernZero} then implies that the projective completion $X_{\phi}$ has canonical class $\mathcal{O}_{\bP}(-2)\otimes \pi^*(-M_{\geq 3})$, which again means that after removing $\Delta_{\infty} \cup \pi^{-1}(M_{\geq 3})$ the open 3-fold $Y_{\phi}$ is a smooth Calabi-Yau\footnote{Since the coefficients in the divisor $M_{\geq 3} = \sum_{p\in \mathrm{Pol}_{\geq 3}(\phi)} \lceil ord(p)/2 \rceil p$ are all $\geq 2$, the holomorphic volume form $\kappa_{\phi}$ again has poles of order at least 2 at infinity.}.  It is moreover affine, and one can work with a symplectic structure which is exact and contact type at infinity, which obviates the need to control rational curves via intersection theory at infinity. 

Lemma \ref{Lem:H3} changes.  We now view the differential $\phi$ as a section of the bundle  $K_S(M)^{\otimes 2}$, where the divisor $M = \sum_{p \in \Pol(\phi)} \lceil ord(p)/2 \rceil p$, which means $\phi$  has simple zeroes at odd order poles, and hence the local system $R^2\pi_*\bZ$ picks up monodromy at odd order poles. The identification of $H_3(Y_{\phi};\bZ)$ with the anti-invariant homology of a spectral double cover goes through only if that cover is branched at both zeroes and odd order poles.  The upshot is that $H_3(Y_{\phi};\bZ)$ is now free of rank
\[
n=6g-6+\sum_{p\in \Pol(\phi)} (ord(p)+1).
\] 
Lemma \ref{Lem:DivisorsIndpt} carries over unchanged.    The construction of symplectic forms, matching spheres, and the grading of the WKB algebra, were all  essentially local arguments on $S$, as were the computations of the contribution of constant holomorphic triangles and non-constant polygons covering WKB cells with multiplicity one. At this stage, one sees that the $A_{\infty}$-structure on the WKB algebra $\scrA$ is given by a potential of the shape $W(T, \epsilon) + W'$ as before.  Finally, the crucial result Theorem \ref{Thm:GLFS1}  of Geiss, Labardini-Fragoso and Schr\"oer  also holds in the non-empty boundary case (with no further constraints on the number of punctures), which enables one to conclude the proof as before.

\begin{Example}  Consider $(\S,\M)$ to be an annulus with $p$ respectively $q$ marked points on the two boundary components.  This corresponds to differentials on $\bP^1$ with poles of order $p+2$ and $q+2$ and $n=p+q$ simple zeroes.  The quiver $Q$ is a non-cyclic orientation of the affine $A_{n-1}$-Dynkin diagram, with trivial potential. The Lagrangian spheres of a WKB collection form a cycle of $n$ matching spheres encircling the origin; the vanishing of the potential reflects the fact that the threefold fibres over $\bC^*$, and there are no non-constant holomorphic polygons (by the maximum principle applied in $\bC^*$).  Note that the poles of the holomorphic volume form and  resulting grading of the endomorphism algebra of the WKB collection depend on the decomposition $n=p+q$. Compare to \cite[Section 12.3]{BrSm}.
\end{Example}

\subsection{Simple poles} \label{Sec:SimplePole}
Whilst one can construct the triangulated category  associated to $(\S,\M)$ by starting with a non-degenerate ideal triangulation defined by a point $\phi$ of the open stratum $\Quad(\S,\M)_0 \subset \Quad(\S,\M)$, the space of stability conditions on $\D\scrK((\S,\M);b))$ is (an orbifolded version of) the larger space $\Quad(\S,\M)$ in which the differentials may acquire simple poles by having zeroes collide with double poles.  The universal family $\scrY \rightarrow \Quad(\S,\M)_0$ 
of threefolds (which exists locally) does \emph{not} extend to $\Quad(\S,\M)$ as a locally trivial fibre bundle.   In any extension of the universal family there must be a non-trivial degeneration of $Y_{\phi}$ when $\phi$ acquires a simple pole, to account for the non-triviality of the monodromy action on $H^2(Y_{\phi};\bZ_2)$ described in Lemma \ref{Lem:MonodromySwaps}, which exchanges components of the reducible fibre. 

Return to the local model
\[
\{ a^2 - \delta\,bc = \phi(\delta) t^2\} \subset \bP^3_{[a,b,c,t]} \times \bC_{\delta} 
\]
for the neighbourhood of the reducible fibre $\{a^2-t^2=0\} \subset \bP^3\times\{0\}$ inside $X_{\phi}$, where $\phi \in H^0(K_S(M)^{\otimes 2})$ has a double pole at $\delta = 0\in S$ and $\phi(0)=1$.  Suppose instead that we have a differential $\phi_0$ with $\phi_0(0) = 0$, so the quadratic differential $\phi_0$ has a simple pole at the corresponding point of $S$.  The corresponding 3-fold 
\[
X_{\phi_0} \ = \ \{a^2 = \delta(bc+t^2\} \subset \bP^3\times\bC_{\delta} 
\]
has a multiple fibre $\{a^2=0\}$, and  is singular along a smooth conic curve $\{a=0=\delta, \, bc+t^2=0\}$.  The transversal singularity is a surface ordinary double point.  

It is tempting to associate to a quadratic differential $\phi_0$ with a simple pole at $p\in M$ a smooth 3-fold
\[
Y_{\phi_0} \rightarrow S
\]
with fibre $T^*\bR\bP^2$, with multiplicity 2, over $p$. For suitable symplectic forms, there is  a Lagrangian $\bR\bP^3 \subset Y_{\phi_0}$, fibred by a Morse-Bott function over an arc from a zero to $p$ (the vanishing cycle of the zero converges at $p$ to the $\bR\bP^2$  Bott minimum). The two Lagrangian spheres depicted on the right side of  Figure \ref{Fig:ImmersedLag}, which should persist in the category defined by $\phi_0$ since this does not change in passing from $\Quad(\S,\M)_0$ to $\Quad(\S,\M)$, would now be realised by this $\bR\bP^3$, equipped with either the trivial or the non-trivial spin structure or equivalently $\bZ_2$-valued local system.    It would be interesting to find a model (of this or another form) for the ``locally trivial family of categories"  over a point of the incomplete locus.

\subsection{No poles}\label{Sec:NoPoles}

Let $S$ be a Riemann surface of genus $\geq 2$. Consider again the 3-fold $Y'_{\psi}$ from \eqref{Eqn:Holo3fold} associated to a holomorphic quadratic differential $\psi$ on $S$ with distinct zeroes.  This has a fibrewise compactification which is a Lefschetz fibration over $S$ with generic fibre $\bP^1 \times \bP^1$.  According to \cite[Corollary 4.9]{Smith:Quadrics}, there is an equivalence of $\bZ_2$-graded split-closed categories 
\begin{equation} \label{Eqn:QuadricS0}
D^{\pi}\scrF(\bP^1\times\bP^1)_{\lambda = 0} \ \simeq \ D^{\pi}\scrF(S^0)
\end{equation}
between the nilpotent summand of the Fukaya category of the even-dimensional quadric, i.e. that corresponding to the zero eigenvalue of quantum product with the first Chern class, and the Fukaya category of a pair of points (both categories are semisimple).  The paper \cite{Smith:Quadrics} works over $\bC$, but the argument underlying the equivalence of \eqref{Eqn:QuadricS0} (which uses the computation of a certain Gromov-Witten invariant counting sections of a Lefschetz fibration with fibre $\bP^1\times \bP^1$ and some elementary deformation theory) goes through over any algebraically closed characteristic zero field.  Starting from this, one can show in the spirit of \cite[Section 5]{Smith:Quadrics} that, over the Novikov field $\Lambda_\bC$,  the $\bZ_2$-graded category underlying $\scrF(Y'_{\psi})$ admits a formal deformation which is split-closed derived equivalent to the $\bZ_2$-graded Fukaya category $\scrF(C)$ of the spectral cover $C \rightarrow S$, i.e. the double cover of $S$ branched at the zeroes of $\psi$.   

That equivalence takes the matching sphere $L_{\gamma}$ considered in this paper to the obvious circle lying over $\gamma$ in the double covering $C$. An essential simple closed curve $\sigma \subset S$ defines a Lagrangian $(S^1 \times S^2) \subset Y'_{\psi}$ and a  disjoint union of two circles in $C$.  The ring structures on the Floer cohomologies of these objects is different (only one has a non-trivial degree zero idempotent), which suggests that  the deformation of categories from $D^{\pi}\scrF(Y'_{\psi})$ to $D^{\pi}\scrF(C)$ induced by the compactification of the fibres is non-trivial. 

\subsection{The untwisted category}\label{Sec:Untwisted}

For each $b\in H^2(Y_{\phi};\bZ_2)$ there is a total morphism algebra $\scrA(T;b) = \oplus_{e, f \in T} HF^*(L_e, L_f)$ associated to a collection of Lagrangian matching spheres indexed by edges of a non-degenerate triangulation $T$.  The quasi-isomorphism type of this algebra \emph{will} depend on the choice of $b$, even though its idempotent pieces $HF(L_e, L_e)$ do not.  
Theorem \ref{Thm:Main2} concerns the case $b=b_0$ and the twisted Fukaya category $\scrF(Y_{\phi};b_0)$, and it is natural to consider the untwisted category $\scrF(Y_{\phi})$ corresponding to taking $b=0$.

 Lemma \ref{Lem:LocalNonzero2} determined the signs with which the two rigid disks with boundary on $\Gamma$ contribute to their moduli space.  These discs cancelled in the category $\scrF(Y_{\phi})$. Therefore, the category $\D(\scrA(T;0))$ is obtained by replacing the potential of \eqref{Eqn:PotentialForNondegenQuiver} by its first term $\sum_f T(f)$, killing the higher $A_{\infty}$-products;  alternatively, it arises if one constructs  the threefold $Y_{\phi}$ by removing the conic fibres over all poles of $\phi$, and not just those of order $>2$.  This category is drastically different to $\D(\scrA;b_0)$, see the  example considered in  Section \ref{Sec:Example}. 

Moreover, in contrast to Remark \ref{Rem:obstructed}, the Lagrangian cylinder $L_{\sigma} \cong S^1\times S^2$   does define an object of  the untwisted category $\scrF(Y_{\phi})$ if one allows unobstructed branes.  The category $\scrF(Y_{\phi})$ cannot be generated by matching spheres:  the Lagrangian cylinder $L_{\sigma}$ can be deformed through non-Hamiltonian deformations (shrinking the loop $\sigma$) to be disjoint from any given matching path in $S-M$, but (in the untwisted category) still with non-trivial Floer cohomology.

\subsection{An example}\label{Sec:Example}

Consider a differential on $\bP^1$ with poles of order $2$ and $4$; this case was considered in \cite[Section 12]{BrSm}.  The relevant quiver has two vertices and no arrows, and the category $\scrK((\S,\M);b_0)$ is generated by two Lagrangian spheres $L_1, L_2$ with $HF(L_1,L_2) = \{0\}$. All objects are isomorphic to direct sums of shifts of the $L_i$. These objects are depicted on the left of Figure \ref{Fig:Fake}, where the black dot denotes the double pole and the higher order pole lies at $\infty \in \bC\bP^1$.  
\begin{figure}[ht]
\begin{center}
\includegraphics[scale=0.3]{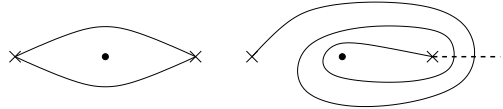}
\caption{The generators of $\scrF(Y_{\phi};b_0)$, and a non-isotopic sphere defining an equivalent object.\label{Fig:Fake}}
\end{center}
\end{figure}
The vanishing of Floer cohomology relies on the choice of the non-trivial background class, which ensures the two rigid disks which project onto the obvious region bound by these arcs count with the same sign: the differential
\begin{equation} \label{Eqn:FloerNontrivial}
\partial: CF^*(L_1,L_2) = k[1]\oplus k[2] \rightarrow k[1]\oplus k[2]  = CF^*(L_1,L_2)
\end{equation}
maps the degree one generator to a non-trivial multiple of the degree two generator. 

On the other hand, there are infinitely many pairwise non-Hamiltonian isotopic Lagrangian spheres in $Y_{\phi}$, distinguished in  the split-closed twisted Fukaya category $\Tw^{\pi}\scrF(Y_{\phi})$  by the rank of Floer cohomology with the non-compact Lefschetz thimble dotted in the right of Figure \ref{Fig:Fake} (with background class $b=0$ the differential in \eqref{Eqn:FloerNontrivial} vanishes, and spherical objects can then arise from non-trivial twisted complexes based on the $L_i$).  Turning on the background class $b_0$ collapses all of these distinct objects onto shifts of $L_1$ and $L_2$ (direct sums are excluded by considering the rank of self Floer cohomology). Therefore isomorphism in the category $\scrF(Y_{\phi};b_0)$ is very far from implying Hamiltonian isotopy.

\subsection{Open directions}

It is natural to wonder if \cite{BrSm} and Theorem \ref{Thm:Main2} have applications to the classical symplectic topology of $Y_{\phi}$, or to representation theory. The results established so far fall slightly short of what seems to be required, although there are several obvious avenues for further study.  

There are natural representations (the first defined by parallel transport)
\begin{equation} \label{mcg}
\pi_1(\Quad^{\pm}(\S,\M)_0) \longrightarrow \pi_0\, \Symp(Y_{\phi};b_0) \longrightarrow \mathrm{Auteq}\, \D\scrF(Y_{\phi};b_0) / \langle [2]\rangle
\end{equation}
where the middle term refers to mapping classes preserving $b_0 \in H^2(Y_{\phi};\bZ_2)$ and 
where on the right we quotient by the square of the shift functor.  It follows from results of \cite{BrSm} that \emph{if} the space of stability conditions $\Stab_{\Delta}(\S,\M)$ studied there is simply-connected, then the first arrow is injective: but simple-connectivity of $\Stab_{\Delta}(\S,\M)$ is currently unknown (see however \cite{YuQiu} for progress in this direction).

When $\partial \S \neq \emptyset$ one can kill the fundamental group of $Y_{\phi}$ by subcritical handle attachments at infinity, to obtain symplectic six-manifolds which are simply-connected but admit symplectomorphisms of positive Floer-theoretic entropy, and which have subgroups of autoequivalences of the Fukaya category which surject onto mapping class groups. This may be of some interest, since whilst such examples are known, they are not yet terribly numerous.

 In another direction, the second map of \eqref{mcg} does surject onto the quotient of the subgroup of autoequivalences which preserve $\Stab_{\Delta}(\S,\M)$ by those which act trivially (``negligible" autoequivalences in the terminology of \cite{BrSm}).  There are many situations in which autoequivalences of derived Fukaya categories have no geometric origin, i.e. do not arise from symplectomorphisms, see e.g. \cite{AbouzaidSmith}.

Finally, the existence of a smooth compactification divisor $\hat{\Delta}_{\infty}$ for $Y_{\phi}$ gives rise to a spectral sequence computing symplectic cohomology $SH^*(Y_{\phi})$,  see \cite{Seidel:bias}. That in turn leads to conjectural bounds on the ranks of the Hochschild cohomology groups of the Ginzburg algebra $A(Q,W)$ which might be of independent interest.  It would be instructive,  in this vein, to relate the wrapped Fukaya category of $Y_{\phi}$ to the derived category of all (not necessarily finite-dimensional) modules over the complete Ginzburg algebra.

\bibliographystyle{amsplain}

\end{document}